\documentclass[12pt,a4paper]{amsart}
\usepackage[italian, UKenglish]{babel}
\usepackage[latin9]{inputenx} 
\usepackage{fontenc}
\usepackage[mathscr]{eucal}
\usepackage{indentfirst}
\usepackage{verbatim}
\usepackage[margin=2.9cm]{geometry}
\usepackage[colorinlistoftodos]{todonotes}
\usepackage[foot]{amsaddr}
\usepackage{xcolor}
\usepackage[framemethod=tikz]{mdframed} 
\usepackage{amsfonts}
\usepackage{amsthm}
\numberwithin{equation}{section}
\usepackage{amsmath}
\usepackage{amscd}
\usepackage{amssymb}
\usepackage{mathtools}
\usepackage{mathabx}
\usepackage{mathptmx}
\usepackage{hyperref} 
\makeatletter
\def\@tocline#1#2#3#4#5#6#7{\relax
  \ifnum #1>\c@tocdepth 
  \else
    \par \addpenalty\@secpenalty\addvspace{#2}%
    \begingroup \hyphenpenalty\@M
    \@ifempty{#4}{%
      \@tempdima\csname r@tocindent\number#1\endcsname\relax
    }{%
      \@tempdima#4\relax
    }%
    \parindent\z@ \leftskip#3\relax \advance\leftskip\@tempdima\relax
    \rightskip\@pnumwidth plus4em \parfillskip-\@pnumwidth
    #5\leavevmode\hskip-\@tempdima
      \ifcase #1
       \or\or \hskip 1em \or \hskip 2em \else \hskip 3em \fi%
      #6\nobreak\relax
    \dotfill\hbox to\@pnumwidth{\@tocpagenum{#7}}\par
    \nobreak
    \endgroup
  \fi}
\makeatother
\newcommand{\mycomment}[1]{%
}

\newcommand*{\norm}[1]{\left\lVert#1\right\rVert}
\newcommand*{\vertiii}[1]{{\left\vert\kern-0.25ex\left\vert\kern-0.25ex\left\vert #1 \right\vert\kern-0.25ex\right\vert\kern-0.25ex\right\vert}}

\makeatletter
\newcommand\mathcircled[1]{%
  \mathpalette\@mathcircled{#1}%
}
\makeatother
\newtheorem{innercustomgeneric}{\customgenericname}
\providecommand{\customgenericname}{}
\newcommand{\newcustomtheorem}[2]{%
  \newenvironment{#1}[1]
  {%
   \renewcommand\customgenericname{#2}%
   \renewcommand\theinnercustomgeneric{##1}%
   \innercustomgeneric
  }
  {\endinnercustomgeneric}
}
\newcustomtheorem{customprop}{Proposition}
\newcustomtheorem{customcor}{Corollary}
\newcustomtheorem{customdefin}{Definition}
\makeatletter
\renewcommand{\email}[2][]{%
  \ifx\emails\@empty\relax\else{\g@addto@macro\emails{,\space}}\fi%
  \@ifnotempty{#1}{\g@addto@macro\emails{\textrm{(#1)}\space}}%
  \g@addto@macro\emails{#2}%
}
\makeatother
 
\allowdisplaybreaks 
\DeclareMathAlphabet{\mathcal}{OMS}{cmsy}{m}{n}

\DeclareMathOperator{\spn}{span}

\theoremstyle{plain}
\newtheorem{theor}{Theorem}[section] 
\newtheorem{lem}[theor]{Lemma} 
\newtheorem{cor}[theor]{Corollary}
\newtheorem{prop}[theor]{Proposition}
\newtheorem*{cor*}{Corollary}
\newtheorem*{prop*}{Proposition}
\newtheorem*{lem*}{Lemma}
\theoremstyle{definition}
\newtheorem{defin}[theor]{Definition}
\newtheorem*{defin*}{Definition}

\newtheorem{ex}[theor]{Example}

\newtheorem{no}[theor]{Notation}

\begin{document}
\title{Horizontal and Straight Triangulation on Heisenberg Groups}
\author[G. Canarecci]{Giovanni Canarecci}
\address{University of Helsinki \\ Department of Mathematics and Statistics \\
Helsinki, Finland} 
\email{giovanni.canarecci@helsinki.fi}
\keywords{Heisenberg group, horizontal, Sub-Riemannian geometry, triangulation, simplexes} 
\begin{abstract} 
This paper aims to show that there exists a triangulation of the Heisenberg group $\mathbb{H}^n$ into singular simplexes with regularity properties on both the low-dimensional and high-dimensional layers. %
For low dimensions, we request our simplexes to be horizontal while, for high dimensions, we define a notion of straight simplexes using exponential and logarithmic maps and we require our simplexes to have high-dimensional straight layers. %
A triangulation with such simplexes is first constructed on a general polyhedral structure and then extended to the whole Heisenberg group. %
In this paper we also provide some explicit examples of grid and triangulations.
\end{abstract}
\maketitle
\tableofcontents


\section*{Acknowledgments} 
I would like to thank my adviser, university lecturer Ilkka Holopainen for his reviews. %
I also want to thank professor Pierre Pansu for the stimulating discussions and suggestions.


\section*{Introduction}

The aim of this paper is to show that there exists a triangulation of the Heisenberg group $\mathbb{H}^n$ into singular simplexes with regularity properties on both the low-dimensional and high-dimensional layers. 
%
%
The Heisenberg group $\mathbb{H}^n$, $n \geq 1$, %
with a group operation *, %
is the $(2n+1)$-dimensional 
connected, simply connected and nilpotent Lie group identified through exponential coordinates with $\mathbb{R}^{2n+1}$. %
Such a group has two important automorphisms playing a role in its geometry: left translations (with respect to a fixed point) %
and anisotropic dilations. 
Additionally, the Heisenberg group is a Carnot group of step $2$ with Lie algebra $\mathfrak{h} = \mathfrak{h}_1 \oplus \mathfrak{h}_2$.  %
The \textit{horizontal} layer $\mathfrak{h}_1$ has a standard orthonormal basis of left invariant vector fields,  
$ X_j=\partial_{x_j} -\frac{1}{2} y_j \partial_t$ and $Y_j=\partial_{y_j} +\frac{1}{2} x_j \partial_t$ for $ j=1,\dots,n$,     %
which hold the core property that $[X_j, Y_j] = \partial_t=:T $ for each $j$. 
$T$ alone spans the second layer $\mathfrak{h}_2$ and is called the \textit{vertical} direction. By definition, the horizontal subbundle changes inclination at every point, allowing movement from any point to any other point following only horizontal paths; this allows to define the Carnot--Carath{\'e}odory distance $d_{cc}$, measured along curves whose tangent vector fields are horizontal. 
%
%
In Section \ref{section:example-grid} we show an example of a grid on the Heisenberg group made by cubes of the kind
\begin{align*}
 Q_{ \epsilon p, \epsilon} &:=   \left \{   w \in \mathbb{H}^n   \ / \ w=q+ \epsilon p \text{ with } q \in \left [ 0,\epsilon \right ]^{2n+1} \right \}.
\end{align*}
with $p \in \mathbb{H}^n$ and $\epsilon > 0$ small. %
The faces and subfaces of these cubes enjoy some limited $\mathbb{H}$-regularity properties:

\begin{prop*}[\ref{prop:1-codim-grid} - \ref{prop:faces_of_faces}]
The interior of each face of $Q_{\epsilon p, \epsilon}$ is a $\mathbb{H}$-regular $1$-codimensional  surface and  each face is a $\mathbb{H}$-regular $1$-codimensional  surface %
 unless the face both is (Euclidean) perpendicular to the $t$-axis and intersects the $t$-axis itself.
If $n > 1$, the interior of each subface of $Q_{\epsilon p, \epsilon}$ is a $\mathbb{H}$-regular $2$-codimensional  surface and each subface is a $\mathbb{H}$-regular $2$-codimensional surface %
 unless the face both is (Euclidean) perpendicular to the $t$-axis and intersects the $t$-axis itself.
\end{prop*}

Nevertheless, this grid is limited because we cannot choose vertices arbitrarily and we are  constrained by the side-length of the faces. Another limit is that, a priori, we do not know whether this grid can support some regularity at each dimension, or get associated with a triangulation of the space $\mathbb{H}^n$. 
%
%
In Section \ref{section:triangulation} we present a triangulation of the Heisenberg group with some regularity properties. %
The main inspiration for this section is the preprint from Balogh, Kozhevnikov and Pansu \cite{BKP} which, in return, relies on Gromov \cite{GROMOV}.  %
%
Specifically, in Subsection \ref{subsec:simplexes} we define a \emph{singular} $k$\emph{-simplex} as a continuous map $\sigma^k$:
$$
\sigma^k : \Delta^k \to \mathbb{H}^n,
$$
where $\Delta^k$ is the standard $k$\emph{-simplex} in $\mathbb{R}^{k+1}$; %
then, if $k>1$, we define its boundary as 
$$ 
\partial \sigma^k := \sum_{i=0}^{k} (-1)^i \sigma_i^{k-1},
$$
where $\sigma_i^{k-1}$ is the singular $(k-1)$-simplex  %
that acts as $ \sigma^k$ on the $i^{th}$ face of $ \Delta^{k}$. %
%
%
In Subsection \ref{subsec:straight-horizontal} we introduce %
\emph{straight} $k$\emph{-simplexes} $\sigma_{p_0, \dots, p_k} : \Delta^k \to \mathbb{H}^n$ (Definition \ref{defin:straight_simplexconstr}) and \emph{horizontal} $k$\emph{-simplexes} (Definition \ref{defin:horizontal_simplexes}). %
In the Heisenberg group at low dimensions, it is natural to ask objects to be horizontal, while such a natural regularity does not extend to high dimensions. Straight simplexes are defined to provide another idea of regularity fitting for high dimensions. %
After this, we construct simplexes %
with horizontality properties on low-dimensions and  linear combinations of straight layers on high-dimensions. %

\begin{prop*}[\ref{lemma:lemma17-general}]
Let $\Delta^k$ denote the standard $k$-simplex in $\mathbb{R}^{k+1}$. %
There exist continuous, $\delta_r$-equivariant and $\tau_q$-equivariant maps
$$
\sigma^h : (\mathbb{H}^n)^{k+1} \to  PL( \Delta^k, \mathbb{H}^n ), \quad k = 0,\dots ,2n+1
$$
such that, for each $\sigma_{p_0, \dots, p_k}^h :  \Delta^k \to \mathbb{H}^n$,
\begin{enumerate}
\item
$\sigma_{p_0, \dots, p_k}^h$ is a singular simplex with vertices $p_0, \dots, p_k \in \mathbb{H}^n$,
\item
if $1 \leq k \leq 2n+1$, the faces and subfaces of $\sigma_{p_0, \dots, p_k}^h$ are $\sigma_{p_0, \dots, \hat p_I, \dots, p_k}^h$'s,  $I \subseteq \{ 0, \dots, k\}$.
\end{enumerate}
Moreover
\begin{enumerate}\setcounter{enumi}{2}
\item
 if $1 \leq k \leq 2n+1$, for $1\leq j\leq n, k$ and  $i_0, \dots, i_j \in \{ 0, \dots, k\}$, $i_0 < \dots < i_j$, $\sigma_{p_{i_0}, \dots, p_{i_j}}^h$'s are horizontal piecewise linear maps,
\item
if $ n+1 \leq k \leq 2n+1$, for $n+1 \leq j \leq k$ and $i_0, \dots, i_j \in \{ 0,\dots, k \}  $,   $i_0 < \dots < i_j$, %
$\sigma_{p_{i_0}, \dots, p_{i_j}}^h$'s are linear combinations of singular simplexes with straight $j$-layers.
\end{enumerate}
\end{prop*}

%
In Subsection \ref{subsec:triangulation-Hn} we give a notion of triangulation for polyhedrons 
by applying a partial order to simplexes (Definition \ref{defin:triangulationQK})  and subsequently we triangulate the whole Heisenberg group: %
\begin{prop*}[\ref{prop:triangulation-heisenberg}]
Let $\Delta^k$ denote the standard $k$-simplex in $\mathbb{R}^{k+1}$. %
There exists a triangulation of the Heisenberg group  $\mathbb{H}^n$ composed of singular simplexes 
$$
\sigma^k  : \Delta^k \to \mathbb{H}^n , \quad k = 0,\dots ,2n+1
$$
 such that
\begin{enumerate}
\item
for $0<k \leq n$, $\sigma^k$'s are horizontal piecewise linear maps.
\item
for $ n+1 \leq k \leq 2n+1$,  $\sigma^k$'s are singular simplexes with straight $k$-layer.
\end{enumerate}
\end{prop*}

\noindent
Although it is not within the purpose of this paper, this triangulation has been constructed with an eye on mass control and the support of currents on it. The triangulation is indeed a good candidate for supporting currents with finite mass and integrability properties. Further references on this subject are, among others, sec.2 in \cite{GCslicing}, 2.1-2.4 in \cite{BKP}, \cite{YOUNG2013}, \cite{YOUNG} and ch.10.3 in \cite{ECHLPT}.


\section{Preliminaries}

In this section we introduce the Heisenberg group $\mathbb{H}^n$, its structure as a Carnot group, the standard bases of vector fields and some regularity definitions. %
There exist many good references for such an introduction and we follow mainly sections 2.1, 2.2, 2.3 and 3.1 in \cite{FSSC} and sections 2.1 and 2.2 in \cite{CDPT}.

\subsection{The Heisenberg Group $\mathbb{H}^n$}\label{defH}

\begin{defin}\label{Heisenberg_Group}             
The $n$-dimensional \emph{Heisenberg Group} $\mathbb{H}^n$, $n \geq 1$, %
with the group operation *, %
is the $(2n+1)$-dimensional 
 connected, simply connected and nilpotent Lie group %
 identified through exponential coordinates with $\mathbb{R}^{2n+1}$, where $*$ is defined as
$$
(x,y,t)*(x',y',t') := \left  (x+x',y+y', t+t'- \frac{1}{2} \langle J
 \begin{pmatrix} 
x \\
y
\end{pmatrix} 
, 
 \begin{pmatrix} 
x' \\
y' 
\end{pmatrix} \rangle_{\mathbb{R}^{2n}} \right  ),
$$
with $x,y,x',y' \in \mathbb{R}^n$, $t,t' \in \mathbb{R}$ and $J=  \begin{pmatrix} 
 0 &  I_n \\
-I_n & 0
\end{pmatrix} $. %
One usually writes $x=(x_1,\dots,x_n) \in \mathbb{R}^n$ and $\mathbb{H}^n= (\mathbb{R}^{2n+1}, * )$ for brevity. %
Furthermore, with a 
computation of the matrix product, we see that
$$
(x,y,t)*(x',y',t') := \left  (x+x',y+y', t+t' + \frac{1}{2} \sum_{j=1}^n \left ( x_j y_j'  -  y_j x_j'  \right ) \right  ).
$$
\end{defin}

\noindent
In the Heisenberg group $\mathbb{H}^n$ there are two important groups of automorphisms; %
the first one is the left translation with a fixed $q \in  \mathbb{H}^n$,
\begin{align*}
\tau_q : \mathbb{H}^n  \to \mathbb{H}^n, \ p \mapsto q*p,
\end{align*}
and the second one is the ($1$-parameter) group of the \emph{anisotropic dilations $\delta_r$}, with $r>0$:
\begin{align*}
\delta_r : \mathbb{H}^n   \to \mathbb{H}^n, \ (x,y,t)  \mapsto (rx,ry,r^2 t).
\end{align*}

\noindent
On the Heisenberg group  $\mathbb{H}^n$ we can define different equivalent distances: the Kor\'anyi and the Carnot--Carath\'eodory distance.

\begin{defin}
\label{norm}
We define the \emph{Kor\'anyi} distance on $\mathbb{H}^n$ by setting, for $p,q \in \mathbb{H}^n$,
$$
d_{\mathbb{H}} (p,q) :=  \norm{ q^{-1}*p }_{\mathbb{H}},  
$$
where $ \norm{ \cdot }_{\mathbb{H}}$ is the \emph{Kor\'anyi}  norm
$  
\norm{(x,y,t)}_{\mathbb{H}}:=\left (  |(x,y)|^4+16t^2  \right )^{\frac{1}{4}},
$  
with $(x,y,t) \in \mathbb{R}^{2n} \times  \mathbb{R} $ and $| \cdot |$ being the Euclidean norm.
\end{defin}

\noindent
The Kor\'anyi distance is left invariant, meaning  
$ 
d_{\mathbb{H}} (p*q,p*q')=d_{\mathbb{H}} (q,q')$  for  $p,q,q' \in \mathbb{H}^n,
$
and homogeneous of degree $1$ with respect to $\delta_r$, meaning 
$
d_\mathbb{H} \left ( \delta_r (p), \delta_r (q)  \right ) = r d_{\mathbb{H}} (p,q) $,  for  $ p,q \in \mathbb{H}^n$ and $ r>0.
$ %
Furthermore, the \emph{Kor\'anyi} distance is equivalent to the \emph{Carnot--Carath\'eodory} distance $d_{cc}$, which is measured along curves whose tangent vector fields are horizontal.


\subsection{Left Invariance and Horizontal Structure on $\mathbb{H}^n$}\label{lefthor}

The standard basis of vector fields in the Heisenberg group $\mathbb{H}^n$ gives it the structure of a Carnot group.

\begin{defin}
\label{XYT}
The standard basis of left invariant vector fields in $\mathbb{H}^n$
consists of the following: 
 $$
\begin{cases}
X_j &:= \partial_{x_j} - \frac{1}{2} y_j\partial_{t} \quad \emph{\emph{ for }}  j=1,\dots,n , \\
Y_j &:= \partial_{y_j} + \frac{1}{2} x_j\partial_{t} \quad \emph{\emph{ for }}  j=1,\dots,n,  \\
T &:= \partial_{t}.
\end{cases}
 $$
\end{defin}

\noindent   
One can observe that $\{ X_1,\dots,X_n,Y_1,\dots,Y_n,T \}$ becomes $\{ \partial_{x_1},\dots, \partial_{x_n}, \partial_{y_1},\dots,\partial_{y_n}, \partial_{t} \}$ at the origin. Another easy observation is that the only non-trivial commutators of the vector fields $X_j,Y_j$ and $T$ are   
$
[X_j,Y_j]=T  $, 
for $j=1,\dots,n.
$   
This immediately tells that all higher-order commutators are zero and that the Heisenberg group is a Carnot group of step $2$. Indeed we can write its Lie algebra $\mathfrak{h}$ as 
$
\mathfrak{h} =\mathfrak{h}_1 \oplus \mathfrak{h}_2,
$ 
with
$$
\mathfrak{h}_1 = \spn \{ X_1,  \ldots, X_n, Y_1, \ldots, Y_n \} \quad \text{and} \quad \mathfrak{h}_2 =\spn \{ T \}.
$$
Conventionally one denotes as $\mathfrak{h}_1$ the space of \emph{horizontal} and $\mathfrak{h}_2$ the space of \emph{vertical vector fields}.   %
Similarly, one denotes a map \emph{horizontal} when its tangent vector fields are horizontal.  %
The vector fields $\{ X_1,\dots,X_n,Y_1,\dots,Y_n\}$ are homogeneous of order $1$ with respect to the dilation $\delta_r, \  r \in \mathbb{R}^+$, i.e.,
$$
X_j (f\circ \delta_r)=r X_j(f)\circ \delta_r \quad  \text{ and }   \quad   Y_j (f\circ \delta_r)=r Y_j(f)\circ \delta_r ,
$$
where $f \in C^1 (U, \mathbb{R} )$, $U\subseteq \mathbb{H}^n$ open and $j=1,\dots,n$. On the other hand, the vector field $T$ is homogeneous of order $2$, i.e.,
$$
T(f\circ \delta_r)=r^2T(f)\circ \delta_r.
$$

\noindent
The vector fields $X_1,\dots,X_n,Y_1,\dots,Y_n,T$ form an orthonormal basis of $\mathfrak{h}$ with a scalar product $\langle \cdot , \cdot \rangle $. In the same way, $X_1,\dots,X_n,Y_1,\dots,Y_n$ form an orthonormal basis of $\mathfrak{h}_1$ with a scalar product $\langle \cdot , \cdot \rangle_H $ defined purely on $\mathfrak{h}_1$.

\noindent
Note that, following 2.1 in \cite{CDPT} and 1.2 \cite{CG}, when in Definition \ref{Heisenberg_Group} we identify $\mathbb{H}^n $ with $ \mathbb{R}^{2n+1}$, we do so by denoting a point in the Lie algebra as
$$
V = x_1 X_1 +  \ldots +x_n X_n + y_1 Y_1+ \ldots + y_n Y_n + t T =  ( x_1 , \ldots  x_n, y_1 \ldots, y_n, t)  \in  \mathfrak{h}
$$
and by  identifying $ ( x_1 , \ldots  x_n, y_1 \ldots, y_n, t) \in \mathbb{R}^{2n+1}$ with $\exp(x_1 X_1 +  \ldots +x_n X_n + y_1 Y_1+ \ldots + y_n Y_n + t T) \in \mathbb{H}^n$, where $\exp$ %
is %
the exponential map (see, for instance, \cite{BB})
\begin{align*}
\exp : \mathfrak{h} &\to \mathbb{H}^n\\
V &\mapsto \gamma(1, V, 0_{\mathbb{H}^n}),
\end{align*}
where the map $\gamma (t) := \gamma(t, V, 0_{\mathbb{H}^n})$ denotes the maximal solution of
$$
\begin{cases}
\dot{\gamma}(t) = V_{\gamma (t)},\\
\gamma (0) = 0_{\mathbb{H}^n}.
\end{cases}
$$
%
%
We denote these \emph{exponential coordinates} and use the notation $ ( x_1 , \ldots  x_n, y_1 \ldots, y_n, t) \in \mathbb{H}^n$. %
By, for example, Theorem 1.2.1 in \cite{CG}, we notice that $\exp$ is a diffeomorphism and we denote its inverse as the logarithmic map $\log$.

\begin{no}\label{notW}
Sometimes it will be useful to consider all the elements of the basis of $\mathfrak{h}$ with one symbol; to do so, we write
$$
\begin{cases}
W_j &:= X_j \quad \text{ for } j=1,\dots,n,\\
W_{n+j} &:= Y_j  \quad \text{ for } j=1,\dots,n,\\
W_{2n+1}&:=T.
\end{cases}
$$
In the same way, the point $(x_1,\dots,x_n,y_1,\dots,y_n,t)$ will be denoted as $(w_1,\dots,w_{2n+1})$.
\end{no}

\begin{no}\label{not:Wderivata}
Remember Notation \ref{notW} and let $j=1,\dots,2n$. Define
$$
\tilde w_{ j}:=
\begin{cases}
w_{n+j}, \quad   &j=1,\dots,n,\\
-w_{j-n}, \quad   & j=n+1,\dots,2n.
\end{cases}
$$
Then we have that
$$
W_j = \partial_{w_j} -\frac{1}{2}  \tilde w_{ j} \partial_t, \quad j=1,\dots,2n.
$$
\end{no}

\noindent
Next we give the definition of Pansu differentiability for maps between two Carnot groups $\mathbb{G}$ and $\mathbb{G}'$. After that, we state it in the special case of $\mathbb{G}=\mathbb{H}^n$ and $\mathbb{G}'=\mathbb{R}$.\\
We denote a function $h : (\mathbb{G},*,\delta) \to (\mathbb{G}',*',\delta')$ \emph{homogeneous} if $h(\delta_r(p))= \delta'_r \left ( h(p) \right )$ for all $r>0$.

\begin{defin}[see \cite{PANSU} and 2.10 in \cite{FSSC}]\label{dGGG}
Consider two Carnot groups $(\mathbb{G},*,\delta)$ and $(\mathbb{G}',*',\delta')$. A function $f: U \to \mathbb{G}'$, $U \subseteq \mathbb{G}$ open, is \emph{P-differentiable} at $p_0 \in U$ if there is a (unique) homogeneous Lie group 
 homomorphism $d_H f_{p_0} : \mathbb{G} \to \mathbb{G}'$ such that
$$
d_H f_{p_0} (p) := \lim\limits_{r \to 0} \delta'_{\frac{1}{r}} \left ( f(p_0)^{-1} *' f(p_0* \delta_r (p) ) \right ),
$$
uniformly for $p$ in compact subsets of $U$.
\end{defin}

\begin{defin}
\label{dHHH} 
Consider a function $f: U \to \mathbb{R}$, $U \subseteq \mathbb{H}^n$ open. $f$ is \emph{P-differentiable} at $p_0 \in U$ if there is a (unique) homogeneous Lie group 
 homomorphism $d_H f_{p_0} : \mathbb{H}^n \to \mathbb{R}$ such that
$$
d_H f_{p_0} (p) := \lim\limits_{r \to 0} \frac{  f \left (p_0* \delta_r (p) \right ) - f(p_0) }{r},
$$
uniformly for $p$ in compact subsets of $U$.
\end{defin}

\noindent
Consider again a function $f:U \to \mathbb{H}^n$, $U\subseteq \mathbb{H}^n$ open, and interpret $\mathbb{H}^n = \mathbb{R}^{2n+1}$ and $f$  in components as $f=(f^1,\dots,f^{2n+1})$,  $f^j:U \to \mathbb{R}$, $j=1,\dots,2n+1$. A straightforward computation shows that, if $f$ is P-differentiable in the sense of Definition \ref{dGGG}, then $f^1,\dots,f^{2n}$ are P-differentiable in the sense of Definition \ref{dHHH}.

\begin{defin}[see 2.11 in \cite{FSSC}]\label{veryfirstnabla} 
Consider a function $f$ P-differentiable  at $p \in U$, $f:U \to \mathbb{R}$, $U\subseteq \mathbb{H}^n$ open. 
The \emph{Heisenberg gradient} or \emph{horizontal gradient} of $f$ at $p$ is defined as
$$
\nabla_\mathbb{H} f(p) := \left ( d_H f_p \right )^* \in \mathfrak{h}_1,
$$
or, equivalently,
$$
\nabla_\mathbb{H} f(p) = \sum_{j=1}^{n} \left [  (X_j f)(p) X_j  + (Y_j f)(p) Y_j  \right ].
$$
\end{defin}

\begin{no}[see 2.12 in \cite{FSSC}]\label{CH1}
Sets of differentiable functions can be defined with respect to the P-differentiability. Consider $ U \subseteq \mathbb{G}$ and $V \subseteq \mathbb{G}'$ open, then  
$C_{\mathbb{H}}^1 (U, V)$ is the function space of continuous functions $f:U \to V $  such that the P-differential $d_H f$ is continuous. %
 $[C_{\mathbb{H}}^1(V,U)]^{k} $ is the set of $k$-uples $f = \{ f_1, \dots, f_k  \}$ such that each $f_i \in C_{\mathbb{H}}^1 (U, V)$ for  $i = 1, \dots, k  $.
\end{no}

\noindent
Finally we give some definitions of regularity for surfaces.

\begin{defin}[see 3.1 in \cite{FSSC}]\label{Hreglow}
Let $1\leq k \leq n$. A subset $S \subseteq \mathbb{H}^n$ is a $\mathbb{H}$-\emph{regular} $k$-\emph{dimensional surface} if for all $p \in S$   there exists  a neighbourhood $ U \in \mathcal{U}_p$, an open set $ V \subseteq \mathbb{R}^k$ and a injective mapping $\varphi : V \to U$, $ \varphi \in [C_{\mathbb{H}}^1(V,U)]^{2n+1} $ with $d_H \varphi $ injective such that $ S \cap U = \varphi (V)$.
\end{defin}

\begin{defin}[see 3.2 in \cite{FSSC}]\label{Hreg}
Let $1\leq k \leq n$. A subset $S \subseteq \mathbb{H}^n$ is a $\mathbb{H}$-\emph{regular} $k$-\emph{codimensional surface} if for all $ p \in S $ there exists a neighbourhood $ U \in \mathcal{U}_p$ and a mapping  $ f : U \to \mathbb{R}^k$, $ f \in [C_{\mathbb{H}}^1(U,\mathbb{R}^k)]^k$, such that  $  {\nabla_\mathbb{H} f_1} \wedge \dots \wedge {\nabla_\mathbb{H} f_k}   \neq 0 $ on   $ U $ and  $  S \cap U = \{ f=0 \} $.
\end{defin}


\section{An example of $\mathbb{H}$-regular grid in $\mathbb{H}^n$}\label{section:example-grid}

In this section we show an example of grid on the Heisenberg group $\mathbb{H}^n$, with $n \geq 1$, composed by $\mathbb{H}$-regular $k$-codimensional surfaces on the first two high dimensional layers of the grid, except of course in the case of $n=1$ which has only one high dimensional layer. The $\mathbb{H}$-regularity of the faces makes them invariant under left translations and anisotropic dilations. Nevertheless, this grid is limited because we cannot choose points to be vertices, but we are constrained by the side-length of the faces and the starting point. Another limit is that, a priori, we do not know whether this grid can support some regularity at each dimension, or get associated with a triangulation %
(simplicial complex) %
of the space $\mathbb{H}^n$.

\begin{defin}\label{defin:cubeep}
Consider a point $p \in \mathbb{H}^n$ and $\epsilon > 0$ small. We define a \textit{cube with starting vertex as} $p$ \textit{and side length} $\epsilon$, and we write $Q_{ p, \epsilon}$, as:
\begin{align*}
 Q_{ p, \epsilon} &:=   \left \{   w \in \mathbb{H}^n   \ / \ w=q+ p \text{ with } q \in \left [ 0,\epsilon \right ]^{2n+1} \right \}.
\end{align*}
\end{defin}

\noindent
The reason for defining a cube starting from a vertex will be apparent later when we will consider a cube starting from the origin, and it will be convenient to have the origin as a vertex instead than, for example, as the center of the cube. \\\\ %
Consider now a similar cube with starting vertex $\epsilon p$ and side length $\epsilon$:
\begin{align*}
 Q_{\epsilon p, \epsilon}
&=  \big \{   w \in \mathbb{H}^n   \ / \ w=(w_1, \dots, w_{2n+1}) \text{ with }  \epsilon p_i \leq w_i \leq \epsilon + \epsilon p_i , \ i \in \left \{ 1, \dots, 2n+1 \right \} \big \}
\end{align*}
with $\epsilon > 0$ and $p \in \mathbb{Z}^{2n+1}$.

\begin{lem}
Consider a cube $ Q_{\epsilon p, \epsilon}$ with starting vertex $\epsilon p$ and side length $\epsilon$, $p \in \mathbb{Z}^{2n+1}$, $\epsilon > 0$.  By construction, we get that:
$$
\bigcup_{p \in \mathbb{Z}^{2n+1}}  Q_{\epsilon p, \epsilon} = \mathbb{H}^{n}, 
$$
and  
$$
Q_{\epsilon p, \epsilon} \cap Q_{\epsilon p', \epsilon} = \partial Q_{\epsilon p, \epsilon} \cap \partial Q_{\epsilon p', \epsilon} \quad \quad \text{ for all } p,p' \in \mathbb{Z}^{2n+1}, \ p \neq p'. 
$$
\end{lem}

\noindent
The cubes  $ Q_{\epsilon p, \epsilon}$'s form a grid for the Heisenberg group $\mathbb{H}^n$. %
For such cubes, we can give the following proposition.

\begin{prop}\label{prop:1-codim-grid}
Consider any cube $Q_{\epsilon p, \epsilon}$ with starting vertex $\epsilon p$ and side length $\epsilon$, $p \in \mathbb{Z}^{2n+1}$, $\epsilon > 0$. The interior of each face of $Q_{\epsilon p, \epsilon}$ is a $\mathbb{H}$-regular $1$-codimensional  surface.\\
Moreover, each face is a $\mathbb{H}$-regular $1$-codimensional  surface unless the face both is (Euclidean) perpendicular to the $t$-axis and intersects the $t$-axis itself.
\end{prop}

\begin{proof}
The cube $Q_{\epsilon p, \epsilon}$ has $2(2n+1)$ different faces and, given $j \in  \left \{ 1, \dots, 2n+1 \right \}$, we can group them into two families as $F_{j}$'s and $E_{j}$'s, with
\begin{align*}
F_{j} :& = Q_{\epsilon   p , \epsilon \ \vert_{w_j = \epsilon p_j}}\\
&=\left \{   w \in \mathbb{H}^n   \ / \ w=(w_1, \dots, w_{2n+1}) \text{ with }  \epsilon p_i \leq w_i \leq  \epsilon + \epsilon p_i \text{ for } i \neq j, \ w_j = \epsilon p_j \right \}
\end{align*}
and
\begin{align*}
 E_{j} :&= Q_{\epsilon   p, \epsilon \ \vert_{ w_j = \epsilon + \epsilon p_j}}\\
&=\left \{   w \in \mathbb{H}^n   \ / \ w=(w_1, \dots, w_{2n+1}) \text{ with }  \epsilon p_i \leq w_i \leq  \epsilon + \epsilon p_i \text{ for } i \neq j, \ w_j = \epsilon + \epsilon p_j \right \}.
\end{align*}
Indeed %
$
\bigcup_{j=1}^{2n+1} \left ( F_j \cup E_j \right ) = \partial Q_{\epsilon p, \epsilon}.
$  %
For each $j \in  \left \{ 1, \dots, 2n+1 \right \}$, we define two functions:
\begin{align*}
&f_{j}:  \mathbb{H}^n \to \mathbb{R}                               & \text{ and } \quad \quad &   g_{j}:  \mathbb{H}^n\to \mathbb{R}   \\
&f_{j}(w_1 , \dots, w_{2n+1}) =\epsilon p_j - w_j                                    &   &   g_{j}(w_1 , \dots, w_{2n+1}) = \epsilon + \epsilon p_j - w_j .
\end{align*}
It follows by their definitions that $F_{j} \subseteq \{ f_{j}=0 \}$ and $E_{j} \subseteq \{ g_{j}=0 \}$ for all $j \in \left \{ 1, \dots, 2n+1 \right \}$.\\
Consider now two cases: $j \in  \left \{ 1, \dots, 2n \right \}$ and $j= 2n+1 $. %
In the first case, we let $j=1,\dots,2n$ and, recalling Notation \ref{not:Wderivata}, we write that
$$
W_j f_{j} =  \left (  \partial_{w_j} - \frac{1}{2} \tilde w_{j} \partial_t  \right ) (\epsilon p_{j} -w_j)  =  -1
$$
and
$$
W_j g_{j} =  \left (  \partial_{w_j} - \frac{1}{2} \tilde w_{j} \partial_t  \right ) (\epsilon + \epsilon p_{j} -w_j) = -1.
$$
This implies that $\nabla_{\mathbb{H}} f_{j}  = \sum_{i=1}^{2n} W_i f_j W_i  \neq 0$ 
and, likewise, $\nabla_{\mathbb{H}} g_{j} \neq 0$ on $\mathbb{H}^n$ 
for all $j\in \left \{ 1, \dots, 2n \right \}$. 
By Definition \ref{Hreg}, it follows that $F_{j}$ and $E_{j}$ are $\mathbb{H}$-regular $1$-codimensional  surfaces for $j\in \left \{ 1, \dots, 2n \right \}$. \\%
In the second case, we let $j=2n+1$. %
This time for $i \in \left \{ 1, \dots, 2n+1 \right \}$ and we get that:
$$
W_i f_{2n+1} =  \left (  \partial_{w_i} - \frac{1}{2} \tilde w_{i} \partial_t  \right ) (\epsilon p_{2n+1} -t)  =  \frac{1}{2} \tilde w_{i}  =
\begin{cases}
  \frac{1}{2}w_{n+i}, \quad   &i=1,\dots,n,\\
-   \frac{1}{2} w_{i-n}, \quad   & i=n+1,\dots,2n,
\end{cases}
$$
and
$$
W_i g_{2n+1} =  \left (  \partial_{w_i} - \frac{1}{2} \tilde w_{i} \partial_t  \right ) (\epsilon + \epsilon p_{2n+1} -t) =  \frac{1}{2} \tilde w_{i}  =
\begin{cases}
  \frac{1}{2}w_{n+i}, \quad   &i=1,\dots,n,\\
-   \frac{1}{2} w_{i-n}, \quad   & i=n+1,\dots,2n.
\end{cases}
$$
We see that the horizontal gradients $\nabla_{\mathbb{H}} f_{2n+1}$ and $\nabla_{\mathbb{H}} g_{2n+1}$ vanish only on the $t$-axis, where $w_i=0$ for $i=1,\dots,2n$.  %
Moreover, from the way we designed our grid, the $t$-axis can intersect the faces $F_{2n+1}$ and $E_{2n+1}$ only at their border.\\
Denoting the interior sets of  $F_{2n+1}$ and $E_{2n+1}$ as $F_{2n+1}^{\mathrm{o}}$ and $E_{2n+1}^{\mathrm{o}}$ respectively, we can say that $\nabla_{\mathbb{H}} f_{2n+1} = \sum_{i=1}^{2n} W_i f_{2n+1} W_i  \neq 0$ on $F_{2n+1}^{\mathrm{o}}$ and $\nabla_{\mathbb{H}} g_{2n+1} \neq 0$ on $E_{2n+1}^{\mathrm{o}}$ respectively. By Definition \ref{Hreg}, the interior sets $F_{2n+1}^{\mathrm{o}}$ and $E_{2n+1}^{\mathrm{o}}$ are $\mathbb{H}$-regular $1$-codimensional surfaces. Moreover, if the faces do not intersect the $t$-axis, with the same argument $F_{2n+1}$ and $E_{2n+1}$ are $\mathbb{H}$-regular $1$-codimensional surfaces and this proves the claim.
\end{proof}

\noindent
Moreover, the borders of the faces $F_j$'s and $E_j$'s are surfaces on their own and with their own regularity, as follows. 

\begin{prop}\label{prop:faces_of_faces}
Consider $n > 1$ and $Q_{\epsilon p, \epsilon}$ any cube with starting vertex $\epsilon p$ and side length $\epsilon$, $p \in \mathbb{Z}^{2n+1}$, $\epsilon > 0$. Furthermore, consider any of its faces, divided into its own subfaces. The interior of each subface is a $\mathbb{H}$-regular $2$-codimensional  surface.\\
Moreover, each subface is a $\mathbb{H}$-regular $2$-codimensional surface unless the face both is (Euclidean) perpendicular to the $t$-axis and intersects the $t$-axis itself.
\end{prop}

\noindent
Note that there is no equivalent to Proposition \ref{prop:faces_of_faces} for $n=1$. %
Indeed,  an equivalent results would need to consider (instead of $2$-codimensional surfaces) $1$-dimensional surfaces as in Definition \ref{Hreglow}, which are horizontal continuously differentiable curves. On the contrary, the $1$-dimensional subsurfaces of $Q_{\epsilon p, \epsilon} \subseteq  \mathbb{H}^1$ are simply segments, not necessarily horizontal.

\begin{proof}
Consider a cube $Q_{\epsilon p, \epsilon}$ with starting vertex $\epsilon p$ and side length $\epsilon$, $p \in \mathbb{Z}^{2n+1}$, $\epsilon > 0$, and its faces $F_j$'s and $E_j$'s, $j \in \left \{ 1, \dots, 2n+1 \right \}$, as in the proof of Proposition \ref{prop:1-codim-grid}. By construction, the boundary of a $1$-codimensional face can itself  be naturally written as the union of $4n$ subfaces. \\
Fix  $j \in \left \{ 1, \dots, 2n+1 \right \}$ and consider $k \in \left \{ 1, \dots, 2n+1 \right \}$ with $k \neq j$. The boundary pieces of $F_j$ can be grouped into two families as $F_{k, F_j}$'s and $E_{k, F_j}$'s:
\begin{align*}
F_{k, F_j} :=\Big \{&   w \in \mathbb{H}^n   \ / \ w=(w_1, \dots, w_{2n+1}) \text{ with }  \epsilon p_i \leq w_i \leq  \epsilon + \epsilon p_i \text{ for } i \neq j,k, \\
&w_j = \epsilon p_j, \ w_k = \epsilon p_k \Big \} 
\end{align*}
and
\begin{align*}
E_{k, F_j} :=\Big \{&   w \in \mathbb{H}^n   \ / \ w=(w_1, \dots, w_{2n+1}) \text{ with }  \epsilon p_i \leq w_i \leq  \epsilon + \epsilon p_i \text{ for } i \neq j,k, \\
&w_j = \epsilon p_j, \ w_k = \epsilon + \epsilon p_k \Big \}. 
\end{align*}
Likewise,  the boundary pieces of $E_j$ can be written as $F_{k, E_j}$'s and $E_{k, E_j}$'s:
\begin{align*}
F_{k, E_j} :=\Big \{&   w \in \mathbb{H}^n   \ / \ w=(w_1, \dots, w_{2n+1}) \text{ with }  \epsilon p_i \leq w_i \leq  \epsilon + \epsilon p_i \text{ for } i \neq j,k, \\
&w_j =\epsilon + \epsilon p_j, \ w_k = \epsilon p_k \Big \} 
\end{align*}
and
\begin{align*}
E_{k, E_j} :=\Big \{&   w \in \mathbb{H}^n   \ / \ w=(w_1, \dots, w_{2n+1}) \text{ with }  \epsilon p_i \leq w_i \leq  \epsilon + \epsilon p_i \text{ for } i \neq j,k, \\
&w_j =\epsilon + \epsilon p_j, \ w_k = \epsilon + \epsilon p_k \Big \}. 
\end{align*}
Indeed %
$
\bigcup_{\substack{k=1 \\ k\neq j}}^{2n+1} \left (F_{k, F_j} \cup E_{k, F_j} \right ) = \partial F_j
$  %
and
$
\bigcup_{\substack{k=1 \\ k\neq j}}^{2n+1} \left (F_{k, E_j} \cup E_{k, E_j} \right ) = \partial E_j .
$ 
For simplicity, we consider $F_{k, F_j}$ to fix the idea. We define the function
\begin{align*}
&h:  \mathbb{H}^n   
\to \mathbb{R}^2,\\
&h(w_1 , \dots, w_{2n+1}) = (\epsilon p_j -w_j, \epsilon p_k -w_k) = (h^1, h^2),
\end{align*}
for which we have that 
$
 F_{k, F_j} = 	   \left \{   h = 0  \right \}
$.\\ %
Consider now different cases. %
In the first case, we take $j,k \in \left \{ 1, \dots, 2n \right \}$, $k \neq j$   
and we have that
$$
\begin{cases}
W_i h^1 = -\delta_{ji} \\
W_i h^2 = -\delta_{ki}
\end{cases}
 \quad  \text{for each } i \in \left \{ 1, \dots, 2n \right \}.
$$
Thus
\begin{align*}
\nabla_{\mathbb{H}} h^1  \wedge \nabla_{\mathbb{H}} h^2 
&= \sum_{l=1}^{2n} W_l h^1 W_l \wedge \sum_{m=1}^{2n} W_m h^2 W_m 
= \sum_{\substack{l,m=1 \\ l\neq m}}^{2n} W_l h^1   W_m h^2 W_l \wedge  W_m\\
&= \sum_{\substack{l,m=1 \\ l\neq m}}^{2n} \delta_{jl}\delta_{km} W_l \wedge  W_m
= W_j \wedge  W_k \neq 0.
\end{align*}
By Definition \ref{Hreg}, it follows that $F_{k, F_j}$ is a $\mathbb{H}$-regular $1$-codimensional  surfaces for $j,k \in \left \{ 1, \dots, 2n \right \}$, $k \neq j$.\\ %
As a second case, we keep  $j \in  \left \{ 1, \dots, 2n \right \}$ and consider the case of $k=2n+1$. Consequently, as above we get
$$
W_i h^1 = -\delta_{ji}  \quad  \text{for each } i \in \left \{ 1, \dots, 2n \right \}  
$$
and then
\begin{align*}
\nabla_{\mathbb{H}} h^1  \wedge \nabla_{\mathbb{H}} h^2 
&= \sum_{l=1}^{2n} W_l h^1 W_l \wedge \sum_{m=1}^{2n} W_m h^2 W_m 
= \sum_{\substack{l,m=1 \\ l\neq m}}^{2n} W_l h^1   W_m h^2 W_l \wedge  W_m\\
&= \sum_{\substack{l,m=1 \\ l\neq m}}^{2n} -\delta_{jl} W_m h^2 W_l \wedge  W_m
= \sum_{\substack{m=1 \\ m\neq j}}^{2n} - W_m h^2 W_j \wedge  W_m
= \sum_{\substack{m=1 \\ m\neq j}}^{2n}  W_m h^2  W_m   \wedge  W_j  ,
\end{align*}
where
$$
W_m h^2 =  \left (  \partial_{w_m} - \frac{1}{2} \tilde w_{m} \partial_t  \right ) (\epsilon p_{2n+1} -t)  =  \frac{1}{2} \tilde w_{m}  =
\begin{cases}
  \frac{1}{2}w_{n+m}, \quad   &m=1,\dots,n,\\
-   \frac{1}{2} w_{m-n}, \quad   & m=n+1,\dots,2n.
\end{cases}
$$
The wedge $\nabla_{\mathbb{H}} h^1  \wedge \nabla_{\mathbb{H}} h^2 $ vanishes only on the $t$-axis, where $w_i=0$ for $i=1,\dots,2n$.  %
Moreover, from the way we designed our grid, the $t$-axis can intersect the face $F_{2n+1, F_j}$ only at its borders.\\
Then we can denote the interior set of  $F_{2n+1, F_j}$ as $F_{2n+1, F_j}^{\mathrm{o}}$ and say that $\nabla_{\mathbb{H}} h^1  \wedge \nabla_{\mathbb{H}} h^2  \neq 0$ on $F_{2n+1, F_j}^{\mathrm{o}}$. By Definition \ref{Hreg}, the interior set $F_{2n+1, F_j}^{\mathrm{o}}$ is a $\mathbb{H}$-regular $2$-codimensional  surface. %
Moreover, if the face does not intersect the t-axis, with the same argument $F_{2n+1, F_j}$ is a $\mathbb{H}$-regular 2-codimensional surfaces.\\
Although $k$ and $j$ are not symmetrical in their geometrical meaning (we denoted a $1$-codimensional face as $F_j$ and a $2$-codimensional face as $F_{k, F_j}$), in the proof their role is interchangeable and the same follows if we consider   $j=2n+1$ with $k \in  \left \{ 1, \dots, 2n \right \}$. This proves the claim.
\end{proof}


\section{A triangulation of the Heisenberg Group $\mathbb{H}^n$}\label{section:triangulation}

In this section we show a triangulation of the Heisenberg group with some regularity properties. %
In Subsection \ref{subsec:simplexes} we present a quick introduction and notations for singular $k$-simplexes, their faces and boundary. %
In Subsection \ref{subsec:straight-horizontal} we define straight simplexes, which we use to construct simplexes  %
with horizontality properties on low-dimensions and  linear combinations of straight layers on high-dimensions. %
%
In Subsection \ref{subsec:triangulation-Hn} we give a notion of triangulation for polyhedrons 
and, subsequently, for the whole Heisenberg group as well.

\subsection{Simplexes  in $\mathbb{H}^n$}\label{subsec:simplexes}
In this subsection we present a quick introduction and notations for singular $k$-simplexes, their faces and boundary. The main source for this part is the  book of Maunder \cite{MAU}.

\begin{defin}\label{defin:standard-simplex}
Let $k \in \mathbb{N}$ and denote $\Delta^k$ the standard $k$\emph{-simplex} in $\mathbb{R}^{k+1}$ with vertices $e_0=(1,0,\dots,0)$, \dots, $e_k=(0,\dots,0,1)$, meaning that we have
\begin{align*}
\Delta^{k} &= \{ (s_1, \dots, s_{k+1}) \ / \ s_1 + \dots + s_{k+1}  = 1, \  s_1, \dots, s_{k+1} \geq 0 \} \\
&= \{ (s_1, \dots, s_k, 1- s_1 - \dots - s_k ) \ / \   s_1, \dots, s_{k} \geq 0, \ s_1 + \dots + s_{k}  \leq 1 \} .
\end{align*}
\end{defin}

\noindent
Note that if we have $\Delta^{k-1}, \Delta^k \in \mathbb{R}^{k+1}$, it immediately follows that
$$
\Delta^{k} \cap \{  s_1 + \dots + s_{k}  = 1 \} = \Delta^{k-1}.
$$
For example,
\begin{align*}
\Delta^{1} &= \{ (s_1, s_2, 0) \ / \ s_1 + s_2  = 1, \  s_1, s_2 \geq 0 \}, \\
\Delta^{2} &= \{ (s_1, s_2, 1- s_1 - s_2 ) \ / \   s_1, s_2 \geq 0, \ s_1 + s_2  \leq 1 \}, \\
\Delta^{2} & \cap \{  s_1 + s_2  = 1 \} = \Delta^{1}.
\end{align*}

\begin{defin}[see 4.2.1-2 in \cite{MAU}]\label{defin:singular-simplex}
Let $k \in \mathbb{N}$ and denote $\Delta^k$ the standard $k$\emph{-simplex} in $\mathbb{R}^{k+1}$ with vertices $e_0=(1,0,\dots,0)$, \dots, $e_k=(0,\dots,0,1)$. %
We define a \emph{singular} $k$\emph{-simplex} in $\mathbb{H}^n$ as a continuous map from the standard $k$-simplex to the Heisenberg group:
$$
\sigma^k : \Delta^k \to \mathbb{H}^n.
$$
\end{defin}

\begin{defin}[see 4.2.3 in \cite{MAU} and pp. 244, 340 in \cite{LEEtopological}]
Let $X$ be a topological space and $k \in \mathbb{N}$, $k>0$. We define the $k^{th}$ \emph{singular chain group} of $X$, and we denote it $S_k(X)$, as the free abelian group with the singular $k$-simplexes  in $X$ as generators. %
Elements of $S_k(X)$ are formal linear combinations of singular $k$-simplexes with integer coefficients, where ``formal'' refers to the structure of $S_k(X)$ as a free abelian group (see p. 340 in \cite{LEEtopological}).
%
\end{defin}

\noindent
The restriction of a singular $k$-simplex $\sigma^k : \Delta^k \to \mathbb{H}^n$ to $\Delta^{k-1}$ is a singular $(k-1)$-simplex. %
Indeed, the boundary of a $k$-simplex can be defined as the sum of $(k-1)$-simplexes as follows:

\begin{defin}[see 4.2.4 and 2.4.5 in \cite{MAU}]\label{defin:boundary-simplex}
Consider $k>0$ and $0 \leq i \leq k$. The $i^{th}$ \emph{face map} $F^i :  \Delta^{k-1} \to  \Delta^k$ is the element of $S_{k-1}(\Delta^k)$ which lands on the face of $\Delta^k$ delimited by the vertices  ($e_0, \dots, \hat{e}_i,  \dots, e_k$), where the hat notation $\ \hat{} \ $ means that that vertex has been omitted. Now let $ \sigma^k : \Delta^k \to \mathbb{H}^n$ be a $k$-simplex and define its boundary as
$$
\partial \sigma^k := \sum_{i=0}^{k} (-1)^i  \sigma^{k} \circ F^i.
$$
One can write 
$$ 
\sigma_i^{k-1} :=   \sigma^{k} \circ F^i \quad \text{and} \quad       \partial \sigma^k := \sum_{i=0}^{k} (-1)^i \sigma_i^{k-1},
$$
where $\sigma_i^{k-1}$ becomes the singular (since it is the composition of continuous functions) $(k-1)$-simplex restriction of $\sigma^k$ to $\Delta^{k-1}$ that acts as $ \sigma^k$ on the $i^{th}$ face of $ \Delta^{k}$.
\end{defin}

\begin{lem}[see 4.2 in \cite{MAU}]
The operator $\partial$ extends to a boundary homomorphism
$$
\partial : S_k(\mathbb{H}^n) \to S_{k-1}(\mathbb{H}^n).
$$
Furthermore, the pair $( S_k(\mathbb{H}^n) , \partial)$ is a chain complex.
\end{lem}

\noindent
By definition, the boundary of a singular $k$-chain is the linear combination of the boundaries of its simplexes.

\subsection{Straight and horizontal simplexes in $\mathbb{H}^n$}\label{subsec:straight-horizontal}
In this subsection we define special singular simplexes, called straight simplexes, and we use them to construct other simplexes  %
with horizontality properties on low-dimensions and  linear combinations of straight layers on high-dimensions. %
The main ideas and concepts are based on and expanded from subsections 2.1 and 5.4 of the preprint  \cite{BKP}.

\begin{defin} 
\label{defin:straight_simplexconstr}
Let $k \in \mathbb{N}$, denote $\Delta^k$ the standard $k$-simplex in $\mathbb{R}^{k+1}$ with vertices $e_0=(1,0,\dots,0)$, \dots, $e_k=(0,\dots,0,1)$ %
and consider $k+1$ points $p_0, \dots, p_k \in \mathbb{H}^n$. We define the \emph{straight} $k$\emph{-simplex}
$$
\sigma_{p_0, \dots, p_k} : \Delta^k \to \mathbb{H}^n
$$
iteratively as follows. %
As $0^{th}$ step, we define the map
\begin{align*}
\sigma_{p_0} : \Delta^0=\{ e_0 \} &\to \{ p_0 \} \subseteq\mathbb{H}^n\\
e_0 & \mapsto p_0
\end{align*}
which sends $e_0$ to $p_0 \in \mathbb{H}^n$. For $1 \leq j \leq k$, the $j^{th}$ step consists of three parts:
\begin{enumerate}
\item
consider the map $ \sigma_{p_0, \dots, p_{j-1}} : \Delta^{j-1} \to \mathbb{H}^n $ from the previous step so that 
$$
\sigma_{p_0, \dots, p_{j-1}} (e_l) = p_l, \quad \text{for } l =0,\dots, j-1
$$
and a point $p_j \in \mathbb{H}^n$. Define a new map:
$$
\bar \Gamma : = \log \circ \tau_{p_j^{-1}} \circ  \sigma_{p_0, \dots, p_{j-1}} : \Delta^{j-1} \to \mathfrak{h}
$$
where %
\begin{itemize}
\item [$\vcenter{\hbox{\tiny$\bullet$}}$]
 $\tau_q :  \mathbb{H}^n  \to \mathbb{H}^n $ is the left translation with respect to the point $q \in \mathbb{H}^n$,
\item [$\vcenter{\hbox{\tiny$\bullet$}}$]
$\log  :  \mathbb{H}^n  \to \mathfrak{h}$ is the logarithmic map defined in subsection \ref{lefthor},
\end{itemize}
\item
extend $\bar \Gamma : \Delta^{j-1} \to \mathfrak{h}$ to a continuous map $\Gamma : \Delta^{j}  \to   \mathfrak{h}$ on the affine cone with vertex at $0 \in \mathfrak{h}$ and base $\bar \Gamma (\Delta^{j-1} )$,
\item
apply the exponential map $\exp :  \mathfrak{h} \to \mathbb{H}^n$ and $\tau_{p_j}  :  \mathbb{H}^n \to \mathbb{H}^n$ 
in order to finish the contruction of $\sigma_{p_0, \dots, p_j}$: 
$$
\sigma_{p_0, \dots, p_j} : =\tau_{p_j} \circ \exp \circ   \Gamma  :  \Delta^{j} \to  \mathbb{H}^n .
$$
\end{enumerate}
\end{defin}

\noindent
Note that in Definition \ref{defin:straight_simplexconstr}, in the $j^{th}$ step:
\begin{itemize}
\item
in part (1), we obtain a set  $\bar \Gamma \left  ( \Delta^{j-1} \right ) \subseteq \mathfrak{h}$, %
which is at most $j-1$ dimensional in $\mathfrak{h}$.
\item
in part (2), the image of $\Gamma$ is a 
cone $\mathcal{C}$ on $\mathfrak{h}$, %
which can be parametrised as
\begin{align*}
\mathcal{C} &= \left \{ \Gamma ( s_1, \dots, s_j, 1 - s_1 - \dots - s_j ) \ / \   s_1, \dots, s_{j} \geq 0, \ s_1 + \dots + s_{j}  \leq 1  \right \} .
\end{align*}
with vertices:
$0, \ \log \left ( p_j^{-1} * p_0 \right ) , \ \dots, \ \log \left ( p_j^{-1} *  p_{j-1}   \right )  .$   %
\item
in part (3), the image $\sigma_{p_0, \dots, p_j} ( \Delta^{j} ) $ has, as expected, vertices $p_0, \dots, p_j$.  
\end{itemize}

\begin{lem}\label{lem:straight-singular}
Straight $k$-simplexes are also singular $k$-simplexes. %
\end{lem}

\begin{proof}
Straight simplexes are  continuous by construction and so they are also singular .
\end{proof}

\begin{no}\label{notation:hat-boundary}
Given Lemma \ref{lem:straight-singular}, we can write the boundary of a straight $k$-simplex $ \sigma_{p_0, \dots, p_k}$ as
$$
\partial \sigma_{p_0, \dots, p_k}= \sum_{i=0}^k (-1)^i  \sigma_{p_0, \dots, p_k} \circ F^i   = \sum_{i=0}^k (-1)^i  \sigma_{p_0, \dots, {\hat{p}}_i , \dots, p_k},
$$
where $\sigma_{p_0, \dots, {\hat{p}}_i , \dots, p_k} :=  \sigma_{p_0, \dots, p_k} \circ F^i$ is the singular $(k-1)$-simplex restriction of $ \sigma_{p_0, \dots, p_k}$ to $\Delta^{k-1}$ that acts as $ \sigma_{p_0, \dots, p_k}$ on the $i^{th}$ face of $ \Delta^{k}$.\\
Similarly we take $I \subseteq \{0,\dots, k\}$,  a singular $k$-simplex $ \sigma_{p_0, \dots, p_k}$ and we shall denote  $\sigma_{p_0, \dots, {\hat{p}}_I , \dots, p_k} $  the singular $(k-|I|)$-simplex restriction of $ \sigma_{p_0, \dots, p_k}$ to $\Delta^{k-|I|}$ that acts as $ \sigma_{p_0, \dots, p_k}$ on the subface of $ \Delta^{k}$ restricted by the indexes in $I$.
\end{no}

\begin{lem}\label{lem:boundary-straight}
Each boundary piece $ \sigma_{p_0, \dots, {\hat{p}}_i , \dots, p_k}$ of a straight $k$-simplex $\sigma_{p_0, \dots, p_k}$ is a straight $(k-1)$-simplex.
\end{lem}

\begin{proof}
Consider $
 \gamma^{(k-1)} =  \sigma_{p_0, \dots, p_k} \circ F^i$  where $  \sigma_{p_0, \dots, p_k}$ is a straight $k$-simplex. In order to prove that $
 \gamma^{(k-1)}: \Delta^{k-1} \to \mathbb{H}^n$ is  a straight simplex, we need to show that it can be constructed following Definition \ref{defin:straight_simplexconstr}. 
In this case its $0^{th}$ step is $k-1=0$ and thus we get $\gamma^{(0)} =   \sigma_{p_0, p_1} \circ F^i: \Delta^{0} = \{1\} \to \mathbb{H}^n$, with $i \in {0,1}$. $F^i$ sends $1$ to the vertex of $\Delta^{1}$ of index not $i$. Namely, for $i = 0$,
\begin{align*}
\gamma^{(0)} =   \sigma_{p_0, p_1} \circ F^0 :   \Delta^{0} &\to \Delta^{1} \quad \quad \quad \ \to \mathbb{H}^n\\
1 &\mapsto  (0,1) = e_1  \mapsto p_1
\end{align*}
or, for $i = 1$,
\begin{align*}
\gamma^{(0)} =   \sigma_{p_0, p_1} \circ F^1 :   \Delta^{0} &\to \Delta^{1} \quad \quad \quad \ \to \mathbb{H}^n\\
1 &\mapsto (1,0) =  e_0  \mapsto   p_0  .
\end{align*}
\noindent
For the $j^{th}$ step, we take %
$\gamma^{(j-1)} = \sigma_{p_0, \dots, p_{j-1}} \circ F^i: \Delta^{j-2} \to \Delta^{j-1} \to  \mathbb{H}^n$ with $i \in \{0,\dots,j-1\}$. 
We know that $ \gamma^{(j-1)} =  \sigma_{p_0, \dots, p_{j-1}} \circ F^i$ reaches the $j-1$ points $p_0, \dots, \hat p_i, \dots, p_{j-1}$, and we are going to add the point $ p_j$. %
We define
\begin{align*}
\bar \Gamma_\gamma  :& \ \Delta^{j-2} \to \mathfrak{h}\\
\bar \Gamma_\gamma :&= \log \circ \tau_{ p_j^{-1}} \circ  \gamma^{(j-1)}\\
&= \log \circ \tau_{ p_j^{-1}} \circ  \sigma_{p_0, \dots, p_{j-1}} \circ F^i  \\
&=  \bar \Gamma  \circ F^i 
\end{align*}
where $ \bar \Gamma $ comes from the $j^{th}$ step of the construction of $\sigma_{p_0, \dots, p_{k}}$.  
From the same construction we also get the map  $ \Gamma :   \Delta^{j}  \to   \mathfrak{h}$ and we use it to define the continuous map
\begin{align*}
\Gamma_\gamma &:   \Delta^{j} \to   \mathfrak{h}\\
\Gamma_\gamma &: = \Gamma \circ F^i,
\end{align*}
which extends $\bar \Gamma_\gamma$  in the same way as $\Gamma$ extends $\bar \Gamma$. %
Finally we take %
$
\tau_{ p_i} \circ \exp \circ   \Gamma_\gamma  :  \Delta^{j} \to  \mathbb{H}^n 
$  %
and we see that
\begin{align*}
\tau_{ p_i} \circ \exp \circ   \Gamma_\gamma 
&=  \tau_{ p_i} \circ \exp \circ     \Gamma  \circ F^i   \\
&=  \sigma_{p_0, \dots, p_{j}}  \circ F^i \\
&=  \gamma^{(j)}
\end{align*}
So the construction is completed, $\gamma^{(k-1)} =  \sigma_{p_0, \dots, p_{k}}  \circ F^i $ is indeed a straight simplex and this completes the claim.
\end{proof}

\begin{defin}\label{defin:straightlayer}
Let $k \in \mathbb{N}$, denote $\Delta^k$ the standard $k$-simplex in $\mathbb{R}^{k+1}$ and consider a  singular $(k-1)$-simplex $ \sigma^{k-1} : \Delta^{k-1} \to \mathbb{H}^n $ with vertices $p_0, \dots, p_{k-1} \in \mathbb{H}^n$. By applying the $j^{th}$ step of Definition \ref{defin:straight_simplexconstr} to $ \sigma^{k-1}$ with a point $p_k \in \mathbb{H}^n $, we obtain a  singular $k$-simplexes 
$$
\sigma^{k} = \sigma_{p_0, \dots, p_{k}} : \Delta^{k} \to \mathbb{H}^n
$$
for which we say that the $k$-\emph{layer of} $\sigma_{p_0, \dots, p_{k}}$ \emph{is straight}.
\end{defin}

\noindent
Note that the crucial difference here is that a straight simplex has all of its layers straight by construction, while a singular simplex with one straight layer says nothing about the other layers.

\begin{defin}\label{defin:horizontal_simplexes}
Let $\Delta^k$ denote the standard $k$-simplex in $\mathbb{R}^{k+1}$ and consider a  singular $k$-simplex $\sigma^{k} $. We say that $\sigma^{k}$ is a \emph{horizontal} $k$\emph{-simplex} if its image is horizontal in $ \mathbb{H}^n$, meaning that the tangent vector fields are horizontal. %
Singular $k$-chains comprised of horizontal $k$-simplexes are called \emph{horizontal} $k$\emph{-chains}.
\end{defin}

\noindent
Notice that %
the spaces of horizontal $k$-simplexes and horizontal $k$-chains are both invariant under dilation  $\delta_r$ and left translation.

\begin{defin}\label{defin:dilationgroup}
Consider $n  \in  \mathbb{N}$  and define the \emph{dilation group} (or \emph{group of dilations}) as
$$
G := \{ \delta_r :  \mathbb{H}^n \to  \mathbb{H}^n , \  r > 0  \}
$$
and the \emph{orbits of the dilation group} as
$$
G(p) := \{ \delta_r (p),   \  r > 0   \} \subseteq  \mathbb{H}^n, \ p \in \mathbb{H}^n.
$$
Furthermore, we say that  $p$ and $q$ are  \emph{on same orbit}, and we write $p \sim q $, if and only if
$$
  \exists r > 0 \text{ s.t. }   \delta_r (p) = q.
$$
This also mean that, if $p, q \neq 0$,
$$
p \in G(q)   \Longleftrightarrow   q \in G(p)   \Longleftrightarrow   p \sim q .
$$
\end{defin}

\begin{no}[see 2.1 in \cite{BKP}]
We denote by $PL(  \Delta^k, \mathbb{H}^n)$ the set of piecewise linear functions from $\Delta^k$ to $\mathbb{H}^n$, where $\mathbb{H}^n$ behaves as a vector space with respect to the exponential coordinates.
\end{no}

\noindent
In the proofs of Proposition \ref{lemma:lemma17} and Proposition \ref{lemma:lemma17-general} we will handle different $k$-simplexes and so use a notation to distinguish between their domains. %
For instance, $2$-simplexes of the kind $\sigma_{p_1, p_2}^h, \ \sigma_{ e, p_2}^h , \  \sigma_{e, p_1}^h $  %
will have domains written as    $ \Delta_{(p_1 , p_2 )}^1 , \  \Delta_{(p_2 , e )}^1 , \  \Delta_{(e , p_1 )}^1 $  %
instead of  $ \Delta^1$.

\begin{prop}[see Lemma 17 in \cite{BKP}]\label{lemma:lemma17}
Let $\Delta^k$ denote the standard $k$-simplex in $\mathbb{R}^{k+1}$ and $p_0, \dots, p_k \in \mathbb{H}^1$. %
There exist continuous maps
$$
\sigma^h : (\mathbb{H}^1)^{k+1} \to  PL( \Delta^k, \mathbb{H}^1 ), \quad k = 0,1,2, 3
$$
such that, for each $\sigma_{p_0, \dots, p_k}^h :  \Delta^k \to \mathbb{H}^1$,
\begin{enumerate}
\item
$\sigma_{p_0, \dots, p_k}^h$ is a singular simplex with vertices $p_0, \dots, p_k$,
\item
if $k\geq 1$, the faces and subfaces of $\sigma_{p_0, \dots, p_k}^h$ are  singular simplexes with vertices $p_0, \dots, \hat p_I, \dots, p_k$,  $I \subseteq \{ 0, \dots, k\}$. 
\end{enumerate}
Moreover
\begin{enumerate}\setcounter{enumi}{2}
\item
if $k \geq 1$, for  $i_0, i_1 \in \{ 0, \dots, k\}$, $i_0 < i_1$, $\sigma_{p_{i_0}, p_{i_1}}^h$'s are horizontal piecewise linear curves.
\item
if $k \in \{2, 3\}$, for $2 \leq j \leq k$ and $i_0, \dots, i_j \in \{ 0,\dots, k \}  $, $i_0 < \dots < i_j $,
$\sigma_{p_{i_0}, \dots, p_{i_j}}^h$   is a linear combination of singular simplexes with straight $j$-layers.
\end{enumerate}
\end{prop}

\noindent
We can say that $\sigma^h$ is $\delta_r$-equivariant and $\tau_q$-equivariant in the sense that, chosen some points $p_0, \dots, p_k$, $\delta_r \circ \sigma_{p_0, \dots, p_k}^h $ and $\tau_q \circ \sigma_{p_0, \dots, p_k}^h $ still satisfy the four points.\\

\noindent
The following proof uses a contact argument 
by Gromov (3.4.B of \cite{GROMOV}). 
A deeper reading on the topic can be found in 3.5 and 4.2 in \cite{GROMOV} or 3.4.3 in \cite{GROMOVPDR}, also by Gromov. %
Another reference on the subject of horizontal triangulations is Section 4 in \cite{YOUNG2013}, which shows two families of examples of groups with horizontal maps and triangulations.

\begin{proof}
$\\$
Points (1) and (2) of the proposition derive immediately by the methods used to prove the following points. %
If $k=0$, the claim is trivial. If $k\geq 1$, consider the point $p_1 \in \mathbb{H}^1$ and let $U \subseteq \mathbb{H}^1$ be a polyhedron containing the origin $e$ in its interior and intersecting each non-trivial orbit of the dilation group, as by Definition \ref{defin:dilationgroup}, exactly once. 
By 
applying 3.4.B in \cite{GROMOV} or Lemma 4.13 in \cite{YOUNG2013}, we can say that, there exists a function
\begin{align*}
\partial U &\to   PL(\Delta^1, \mathbb{H}^1 )\\
 p_1 &\mapsto    \sigma_{e, p_1}^h
\end{align*}
such that $ \sigma_{e, p_1}^h : \Delta^1 \to \mathbb{H}^1$ is a piecewise linear horizontal map that  joins $e$ to $p_1$. %
By dilating and left translating, we define a continuous map
\begin{align*}
\sigma^h :  \quad  (\mathbb{H}^1)^2 &\to  PL( \Delta^1, \mathbb{H}^1 ),\\
(p_0, p_1)&\mapsto  \sigma_{p_0, p_1}^h
\end{align*}
which completes the claim for $k=1$. 
For $k=2$, %
similarly we continue and take a different point $p_2 \in U$, $p_2 \not\in \partial U$, $p_2$ not on the same dilation group orbit as $p_1$ (possible thanks to the definition of $U$) and we get $  \sigma_{ p_1, p_2}^h,  \sigma_{e, p_2}^h \in PL(\Delta^1, \mathbb{H}^n )$ horizontal maps. This satisfies point (3) for $k=2$. 
Notice that we can consider the map $\partial \Delta^2  \to  \mathbb{H}^1 $ defined  by  $ \sigma_{e, p_1}^h,  \sigma_{ p_1, p_2}^h,  \sigma_{e, p_2}^h$ along the border of $\Delta^2$:
$$
\sigma_{p_1, p_2}^h  -  \sigma_{ e, p_2}^h  +  \sigma_{e, p_1}^h : \Delta_{(p_1 , p_2 )}^1 \cup \Delta_{(p_2 , e )}^1 \cup \Delta_{(e , p_1 )}^1 =  \partial \Delta^2  \to  \mathbb{H}^1   
$$
and this works perfectly as the border of a $2$-simplex. 
In order to construct such $2$-simplex, we take the exponential center of gravity 
$$
q = \exp \left (  \frac{1}{3} \left ( \log p_1 + \log p_2  \right ) \right )
$$ %
and we extend all three simplexes $ \sigma_{e, p_1}^h,  \sigma_{ p_1, p_2}^h,  \sigma_{e, p_2}^h$ 
following exactly the $j^{th}$ step of the construction in Definition \ref{defin:straight_simplexconstr}. %
Hence we obtain three singular $2$-simplexes with straight $2$-layers (see Definition \ref{defin:straightlayer}).
\begin{align*}
\sigma_{p_1, p_2, q}^h &: \Delta_{(p_1, p_2, q)}^2 \to \mathbb{H}^1,\\
\sigma_{e, p_2, q}^h &: \Delta_{(e, p_2, q)}^2 \to \mathbb{H}^1,\\
\sigma_{e, p_1, q}^h &: \Delta_{(e, p_1, q)}^2 \to \mathbb{H}^1.
\end{align*}
These three simplexes have the vertex $q$ in common and they all share, two by two, a second vertex as well. Since the construction is exactly the same for all of them, this means that they also perfectly touch on their new borders, meaning that $\sigma_{p_1, p_2, q}^h =  \sigma_{e, p_2, q}^h $ on the part of $\partial \Delta^2$ that joins $p_2$ to $q$, and same for the others.\\
This means that we can define a new $2$-simplex $: \Delta^2 \to \mathbb{H}^1$ where $ \Delta^2$ is divided in three triangles by its center of gravity $\left (\frac{1}{3}, \frac{1}{3}, \frac{1}{3} \right ) \in \Delta^2$ and each of these three triangles is sent to one  of the $2$-simplexes $\Delta_{(*,*,*)}^2$'s in a natural way, with  $\left (\frac{1}{3}, \frac{1}{3}, \frac{1}{3} \right )$ pointing at $q$. %
Hence we obtain a new map
$$
\sigma_{e, p_1, p_2}^h : \Delta^2 \to 
\Delta_{(p_1, p_2, q)}^2 \cup  \Delta_{(e, p_2, q)}^2 \cup  \Delta_{(e, p_1, q)}^2
\to \mathbb{H}^1
$$
which is a $2$-simplex with vertices $e, p_1, p_2$. %
By dilating and left translating, we define a continuous map
\begin{align*}
\sigma^h :  \quad \quad \ (\mathbb{H}^1)^3 &\to  PL( \Delta^2, \mathbb{H}^1 ),\\
(p_0, p_1, p_2)&\mapsto  \sigma_{p_0, p_1, p_2}^h
\end{align*}
which completes the claim for $k=2$. %
For $k=3$, one can work in the same way and so first obtain six horizontal singular $1$-simplexes of the kind  $ \sigma_{p_i, p_j}^h : \Delta^1 \to \mathbb{H}^1$, with $i, j \in \{0,1,2,3\}$, satisfying point (3) for $k=3$. %
Then we obtain four singular $2$-simplexes with straight $2$-layers: %
$ \sigma_{p_i, p_j, p_l}^h : \Delta^2 \to \mathbb{H}^1$, with $i, j, l \in \{0,1,2,3\}$. %
At this point one can see that the map   %
$$
\sigma_{p_1, p_2, p_3}^h  -  \sigma_{ p_0, p_2, p_3}^h  +  \sigma_{p_0, p_1, p_3}^h   -  \sigma_{ p_0, p_1, p_2}^h: %
\Delta_{(p_1 , p_2, p_3 )}^2 \cup \ldots \cup \Delta_{( p_0 , p_1, p_2 )}^2 =  \partial \Delta^3  \to  \mathbb{H}^1   
$$
works perfectly as the border of a $3$-simplex. Again we build this  $3$-simplex by taking the center of gravity
$$
q = \exp \left (  \frac{1}{4} \left ( \log p_0 + \log p_1 + \log p_2 + \log p_3 \right ) \right )
$$ %
and by extending all four simplexes $ \sigma_{p_i, p_j, p_l, p_m}^h$ by one dimension following exactly the $j^{th}$ step of the construction in Definition \ref{defin:straight_simplexconstr}. %
Hence we obtain four singular $3$-simplexes with straight $3$-layers
\begin{align*}
\sigma_{p_1, p_2, p_3, q}^h &: \Delta_{(p_1, p_2, p_3, q)}^3 \to \mathbb{H}^1,\\
\sigma_{p_0, p_2, p_3, q}^h &: \Delta_{(p_0, p_2, p_3, q)}^3 \to \mathbb{H}^1,\\
\sigma_{p_0, p_1, p_3, q}^h &: \Delta_{(p_0, p_1, p_3, q)}^3 \to \mathbb{H}^1,\\
\sigma_{p_0, p_1, p_2, q}^h &: \Delta_{(p_0, p_1, p_2, q)}^3 \to \mathbb{H}^1.
\end{align*}
All these simplexes have the vertex $q$ in common and they all share, two by two, two other vertices as well. Since the construction is exactly the same for all of them, this means that they also perfectly touch on their new borders, meaning that $\sigma_{p_1, p_2, p_3, q}^h =  \sigma_{p_0, p_2, p_3, q}^h $ on the part of $\partial \Delta^3$ that joins $  \sigma_{ p_2, p_3}^h $ to $q$, and same for the others.\\
This means that we can define a new $3$-simplex $: \Delta^3 \to \mathbb{H}^1$ where $ \Delta^3$ is divided in four tetrahedra by its center of gravity $\bar q \in \Delta^3$ and each of these four tetrahedra is sent to one  of the $3$-simplexes $\Delta_{(*,*,*,*)}^3$'s in a natural way, with  $\bar q$ pointing at $q$. %
Hence we obtain a new map
$$
\sigma_{p_0, p_1, p_2,p_3}^h : \Delta^3 \to 
\Delta_{(p_1, p_2, p_3, q)}^3  \cup  \Delta_{(p_0, p_2, p_3, q)}^3  \cup  \Delta_{(p_0, p_1, p_3, q)}^3 \cup \Delta_{(p_0, p_1, p_2, q)}^3 
\to \mathbb{H}^1
$$
which is a $3$-simplex with vertices $p_0, p_1, p_2, p_3$.   %
By dilating and left translating, we define a continuous map
\begin{align*}
\sigma^h :  \quad \quad \ (\mathbb{H}^1)^4 &\to  PL( \Delta^3, \mathbb{H}^1 ),\\
(p_0, p_1, p_2, p_3)&\mapsto  \sigma_{p_0, p_1, p_2, p_3}^h
\end{align*}
which completes the claim for $k=3$ and finishes the proof. %
\end{proof}

\begin{prop}\label{lemma:lemma17-general}
Let $\Delta^k$ denote the standard $k$-simplex in $\mathbb{R}^{k+1}$. %
There exist continuous maps
$$
\sigma^h : (\mathbb{H}^n)^{k+1} \to  PL( \Delta^k, \mathbb{H}^n ), \quad k = 0,\dots ,2n+1
$$
such that, for each $\sigma_{p_0, \dots, p_k}^h :  \Delta^k \to \mathbb{H}^n$,
\begin{enumerate}
\item
$\sigma_{p_0, \dots, p_k}^h$ is a singular simplex with vertices $p_0, \dots, p_k \in \mathbb{H}^n$,
\item
if $1 \leq k \leq 2n+1$, the faces and subfaces of $\sigma_{p_0, \dots, p_k}^h$ are $\sigma_{p_0, \dots, \hat p_I, \dots, p_k}^h$'s,  $I \subseteq \{ 0, \dots, k\}$.
\end{enumerate}
Moreover
\begin{enumerate}\setcounter{enumi}{2}
\item
 if $1 \leq k \leq 2n+1$, for $1\leq j\leq n, k$ and  $i_0, \dots, i_j \in \{ 0, \dots, k\}$, $i_0 < \dots < i_j$, $\sigma_{p_{i_0}, \dots, p_{i_j}}^h$'s are horizontal piecewise linear maps,
\item
if $ n+1 \leq k \leq 2n+1$, for $n+1 \leq j \leq k$ and $i_0, \dots, i_j \in \{ 0,\dots, k \}  $,   $i_0 < \dots < i_j$, %
$\sigma_{p_{i_0}, \dots, p_{i_j}}^h$'s are linear combination of singular simplexes with straight $j$-layers.
\end{enumerate}
\end{prop}

\noindent
As in Proposition \ref{lemma:lemma17}, we can say that $\sigma^h$ is $\delta_r$-equivariant and $\tau_q$-equivariant in the sense that, chosen some points $p_0, \dots, p_k$, $\delta_r \circ \sigma_{p_0, \dots, p_k}^h $ and $\tau_q \circ \sigma_{p_0, \dots, p_k}^h $ still satisfy the four points.

\begin{proof}
$\\$
Points (1) and (2) of the proposition derive immediately by the methods used to prove the following points. %
If $k=0$, the claim is trivial. If $k\geq 1$, %
by 
applying 3.4.B in \cite{GROMOV} or Lemma 4.13 in \cite{YOUNG2013}, and by dilating and left translating, we can say that, there exist a continuous map
\begin{align*}
\sigma^h :  \quad  (\mathbb{H}^n)^{k+1} &\to  PL( \Delta^k, \mathbb{H}^n ),\\
(p_0,\dots, p_k) &\mapsto  \sigma_{p_0,\dots, p_k}^h
\end{align*}
such that each face and subface $ \sigma_{p_{i_0}, \dots, p_{i_j}}^h : \Delta^j \to \mathbb{H}^n$, with    $j\leq n, k$ and  $i_0,\dots i_j \in \{ 0, \dots, k\}$, is a piecewise linear horizontal map with vertices $p_{i_0}, \dots, p_{i_j}$, 
which 
satisfies point (3)   for $ 1 \leq k \leq  2n+1$. In particular this completes the proof for $k \leq n$.  
For $ n+1 \leq k \leq  2n+1$,  this proves point (3) and we are left to verify point (4). %
In order to do so,   
we start with the case $k=n+1$, meaning also $j=n+1$. 
Among the simplexes we just created, we consider the  ${k+1 \choose j}={n+2 \choose n+1} = n+2$ horizontal maps
$\sigma_{p_0, \dots, {\hat{p}}_i , \dots, p_{n+1}}^h \in PL(\Delta^n, \mathbb{H}^n)$, with $i \in \{0,\dots,n+1\}$ where we exclude one point with the hat notation $\ \hat{} \ $ . 
Notice that we can consider the map $\partial \Delta^{n+1}  \to  \mathbb{H}^n $ defined  by  
these $n+2$ maps along the border of $\Delta^{n+1}$:
$$
\sum_{i=0}^{n+1} (-1)^i  \sigma_{p_0, \dots, {\hat{p}}_i , \dots, p_{n+1}}^h : \bigcup\limits_{i=0}^{n+1} \Delta_{(p_0, \dots, {\hat{p}}_i , \dots, p_{n+1})}^{n}  =  \partial \Delta^{n+1}  \to  \mathbb{H}^n
$$
and this works perfectly as the border of a $(n+1)$-simplex according to Notation \ref{notation:hat-boundary}.  
In order to construct such $(n+1)$-simplex, we take the exponential center of gravity 
$$
q = \exp \left (  \frac{1}{n+2} \sum_{i=0}^{n+1}  \log p_{i}   \right )
$$ %
and we extend all $n+2$ singular $n$-simplexes $\sigma_{p_0, \dots, {\hat{p}}_i , \dots, p_{n+1}}^h$'s  by one dimension to $n+2$ singular $(n+1)$-simplexes following 
the $j^{th}$ step of the construction in Definition \ref{defin:straight_simplexconstr}. Hence we obtain $n+2$ singular $(n+1)$-simplexes %
whose $(n+1)$-layer is straight:
$$
\sigma_{p_0, \dots, {\hat{p}}_i , \dots, p_{n+1}, q}^h :   \Delta^{n+1}_{p_0, \dots, {\hat{p}}_i , \dots, p_{n+1}, q}  \to  \mathbb{H}^n
$$
All these simplexes have the vertex $q$ in common and they all share, two by two, $n$ other vertices as well. Since the construction is exactly the same for all of them, this means that they also perfectly touch and coincide on their new borders.\\
%
This means that we can define a new singular $(n+1)$-simplex $: \Delta^{n+1} \to \mathbb{H}^n$ where $ \Delta^{n+1}$ is divided in $n+2$ tetrahedra by its center of gravity in $ \bar q \in \Delta^{n+1}$ and each of those tetrahedra is sent to one  of the $n+2$ singular $(n+1)$-simplexes $ \Delta^{n+1}_{p_0, \dots, {\hat{p}}_i , \dots, p_{n+1}, q}$'s   
 in a natural way, with  $\bar q  \in \Delta^{n+1}$ pointing at $q \in \mathbb{H}^n$. %
Hence we obtain a new map
$$
\sigma_{p_0, \dots, p_{n+1}}^h  :
 \Delta^{n+1}   \to
\bigcup\limits_{i=0}^n \Delta_{(p_0, \dots, {\hat{p}}_i , \dots, p_{n+1}, q)}^{n+1} 
\to \mathbb{H}^n
$$
which is a $(n+1)$-simplex with vertices $p_0,  \dots, p_{n+1}$ and which is a linear combination of singular $(n+1)$-simplexes with straight $(n+1)$-layer.    %
By dilating and left translating, we define a continuous map
\begin{align*}
\sigma^h :  \quad \quad \ (\mathbb{H}^n)^{n+2} &\to  PL( \Delta^{n+1}, \mathbb{H}^n ),\\
(p_0, \dots, p_{n+1})&\mapsto  \sigma_{p_0, \dots, p_{n+1}}^h 
\end{align*}
which completes the claim for $k=n+1$.  
To verify now point (4) for $ n+1 < k \leq  2n+1$ and  $n+1 \leq j\leq k$, one can proceed  in exactly the same manner as we did in the case $k=n+1$ and by induction on $j$,  $n+1 \leq j\leq k$. The base case is then $j=n+1$ and then one can proceed in this manner: first consider the ${k+1 \choose j } = {k+1 \choose n+1 }$ horizontal maps whose linear combination becomes the border of our simplex-to-be. Then extend those $n$-simplexes into $(n+1)$-simplexes adding as point the exponential center of gravity of the current points and following the $j^{th}$ step of the construction in Definition \ref{defin:straight_simplexconstr}. The linear combination of these singular $(n+1)$-simplexes is then a new $(n+1)$-simplex with straight $(n+1)$-layer.  \\%
The main case with $j$ general, supposes that we built straight layers on our simplexes up to dimension $j-1$. In this case one proceed virtually in an identical way: first consider the ${k+1 \choose j } $ singular $(j-1)$-simplexes with straight $(j-1)$-layer whose linear combination becomes the border or our simplex-to-be. Then extend those $(j-1)$-simplexes into $j$-simplexes adding as point the exponential center of gravity of the current points and following the $j^{th}$ step of the construction in Definition \ref{defin:straight_simplexconstr}. The linear combination of these $j$-simplexes is then the new $j$-simplex 
that %
completes the claim.
\end{proof}

\subsection{Regular triangulations of 
 $\mathbb{H}^{n}$}\label{subsec:triangulation-Hn}

In this subsection we give a notion of triangulation for polyhedrons 
and, subsequently, for the whole space as well. %
We start considering a polyhedron $Q$ with vertices in $\{0,1\}^k$, which is the same as saying that we index its vertices by strings in $\{0,1\}^k$, and then we move to a more general case, defining a notion of triangulation and showing that we can use simplexes from Subsetion \ref{subsec:straight-horizontal} to triangulate a polyhedron and then the entire space.

\begin{lem}\label{lemma:injective}
Let $k \in \mathbb{N}$ and denote $\Delta^k$ the standard $k$\emph{-simplex} in $\mathbb{R}^{k+1}$. 
Any injective map %
$$
s: \{ 0,\dots, k\} \to \{0,1\}^k
$$ 
defines a  %
singular $k$-simplex $$\sigma_s : \Delta^k \to Q, $$ 
%
such that $\sigma_s (e_j) =s(j)$   for all  $ j \in \{ 0,\dots, k\} $ and where $Q \subseteq \mathbb{H}^{n}$ is a polyhedron with vertices in $\{0,1\}^k$. %
In particular this means that the vertices of $\sigma_s ( \Delta^k)$ coincide with vertices of $Q$.  
Moreover, we can take  $\sigma_s$ to be a straight  $k$-simplex.
\end{lem}

\begin{proof}
We can take $\sigma_s$ to be the affine map that maps the $k+1$ vertices of the simplex $ \Delta^k$ into the image of $s$, which is comprised by vertices of $Q$, 
 continuously. This gives us a singular simplex. \\
Moreover, since $\sigma_s$ is defined starting from its vertices, we can also construct it following Definition \ref{defin:straight_simplexconstr} and obtain $\sigma_s$ to be a straight $k$-simplex  $\sigma_{s(0), \dots, s(k)}$.
\end{proof}

\noindent
Note that in Lemma \ref{lemma:injective}, if we take $\sigma_s$ to be the affine map,  we get that $\sigma_s ( \Delta^k)$ is contained in the unit $k$-cube %
$[0,1]^k$:
$\sigma_s ( \Delta^k) \subseteq [0,1]^k$.   %
A polyhedron $Q$ covered by simplexes can be called a polyhedral simplex.\\  %

\noindent
In the next corollary, we generalise the set $\{0,1\}^k$ to a set of $2^k$ elements.

\begin{cor}\label{cor:injective-points}
Let $k \in \mathbb{N}$ and denote $\Delta^k$ the standard $k$\emph{-simplex} in $\mathbb{R}^{k+1}$. %
Consider $K \subseteq  \mathbb{H}^n$ a set containing $2^k$ points in  $ \mathbb{H}^n$. %
Any injective map %
$$
s: \{ 0,\dots, k\} \to K, 
$$ 
defines a   singular $k$-simplex  
$$
\sigma_s : \Delta^k \to Q_K, 
$$
such that $\sigma_s (e_j) =s(j)$   for all  $ j \in \{ 0,\dots, k\} $ and where $Q_K \subseteq \mathbb{H}^{n}$ is a polyhedron 
with vertices in $K$. %
In particular this means that the vertices of $\sigma_s ( \Delta^k)$ are $s(0), \dots, s(k)$ and coincide with vertices of $Q_K$.  
Moreover, we can take  $\sigma_s$ to be a straight  $k$-simplex.
\end{cor}

\noindent
The proof of  Corollary \ref{cor:injective-points} is virtually identical to the one for  Lemma \ref{lemma:injective}, just considering the vertices of $Q_K$ instead of $Q$.

\begin{defin}
Let $k \in \mathbb{N}$. %
We say that, for all  $x,y \in \{0,1\}^k$,
$$
x \leq y \quad \text{if and only if} \quad x_i \leq y_i  \text{ for all } i = 0,\dots, k .
$$
This is a partial order on strings and we denote $I$ the set of increasing functions $s$ of the kind $s: \{ 0,\dots, k\} \to \{0,1\}^k$. 
Likewise, consider $2^k$ points in  $ \mathbb{H}^n$ 
 and group them into the set $K$. %
We say that, for all  $x,y \in K$,
$$
x \leq y \quad \text{if and only if} \quad x_i \leq y_i  \text{ for all } i = 1,\dots, 2n+1 .
$$
This is a partial order on points of $\mathbb{H}^n$ and we denote $I_K$ the set of increasing functions $s$ of the kind $s: \{ 0,\dots, k\} \to K$.
\end{defin}

\noindent
Note that, if $K = \{0,1\}^{2n+1}$ and we associate strings in $\{0,1\}^{2n+1}$ to the corresponding points in $ \mathbb{H}^n$,  the two partial orders are the same.

\begin{defin}[see 2.2 in \cite{BKP}]
Consider a  singular $k$-simplex $\sigma_s : \Delta^k \to Q_K 
$ coming from an injective map $s$ as above. The orientation of $\sigma_s$  is defined as: 
$$
(-1)^s:= 
sign \left ( \det 
\begin{bmatrix}
          s(1)-s(0) \\
           \vdots \\
           s(k)-s(0)
\end{bmatrix}
\right ).
$$
\end{defin}

\begin{defin}[see 2.2 in \cite{BKP}]\label{defin:triangulationQK}
The set of singular $k$-simplexes $\sigma_s  : \Delta^k \to Q_K $,  
such that 
$s \in I_K$, establishes a \emph{triangulation of} $Q_K$. 
Furthermore, we denote a  \emph{singular $k$-chain of the triangulation} by
$$
\tau := \sum_{s \in I} (-1)^s \sigma_{s}.
$$
Similarly, by taking straight $k$-simplexes $\sigma_s : \Delta^k \to Q_K$, such that $s \in I_K$, we get a \emph{straight triangulation of} $Q_K 
$ and  %
a \emph{straight $k$-chain of the triangulation}.
\end{defin}

\noindent
Note that singular $k$-simplexes $\sigma_s$'s comprised in a triangulation do not overlap each other except that at the border because they are built following the same construction, whether as affine maps, straight simplexes or else. 
This also means that the border of $\tau$, which is a linear combination of the borders of its simplexes, lives on the border of $Q_K$, 
since the internal borders of the simplexes will cancel each others out:
$$
\partial \tau = \sum_{s \in I} (-1)^s  \partial \sigma_{s}=
\sum_{s \in I} (-1)^s   \sum_{i=0}^{k} (-1)^i  \sigma_{s} \circ F^i.
$$

\noindent
Furthermore note that, since the singular simplexes are continuous and by construction, this definition of triangulation is consistent with, for example, the one in Section 2.2 in \cite{YOUNG} and Section 1.3.2 in \cite{ABB}.

\noindent
After Definition \ref{defin:triangulationQK}, one can see that $k!$ simplexes are needed to triangulate a $k$-dimensional cube, meaning that $|I| = k!$.  %
Moreover, recalling Definition \ref{defin:boundary-simplex} about the border of a simplex, one can say that the faces (or subfaces of any dimension) of $Q_K$ 
are naturally triangulated by the boundary pieces of the higher-dimensional simplexes belonging to the triangulation of the cube. \\%
Last, we can generalise the triangulation to the whole $\mathbb{H}^n$ via general polyhedrons. 
For example,  we could triangulate the grid of cubes $ Q_{\epsilon p, \epsilon}$'s from Section \ref{section:example-grid} by using singular $k$-simplexes as affine maps, which provides $\mathbb{H}^n$ with a triangulation into singular simplexes for the whole space, albeit with very little regularity (not straight nor horizontal).  
%
Another example would be to triangulate polyhedrons 
with vertices in $\mathbb{Z}^{2n+1}$ as straight simplexes using Definition  \ref{defin:straight_simplexconstr}, which would give us a triangulation of straight  simplexes for the whole space  $\mathbb{H}^n$ (but with no horizontality for low dimensions).  %
In the following proposition we propose a yet different kind of triangulation.

\begin{prop}\label{prop:triangulation-heisenberg}
Let $\Delta^k$ denote the standard $k$-simplex in $\mathbb{R}^{k+1}$. %
There exists a triangulation of the Heisenberg group  $\mathbb{H}^n$ composed of singular simplexes 
$$
\sigma^k  : \Delta^k \to \mathbb{H}^n , \quad k = 0,\dots ,2n+1
$$
 such that
\begin{enumerate}
\item
for $0<k \leq n$, $\sigma^k$'s are horizontal piecewise linear maps.
\item
for $ n+1 \leq k \leq 2n+1$,  $\sigma^k$'s are singular simplexes with straight $k$-layer.
\end{enumerate}
\end{prop}

\begin{proof}
Thanks to Proposition \ref{lemma:lemma17-general} we have families of suitable simplexes. By grouping them into triangulations of polyhedrons starting from a grid of vertices in $\mathbb{Z}^{2n+1}$, we get a triangulation of the whole space.
\end{proof}

\noindent
In Appendix \ref{appendix:extra-example} we show two basic examples, triangulating the square and the cube in $\mathbb{H}^1$.


\renewcommand{\thesection}{\Alph{section}}
\section{Appendix: examples of triangulations}\label{appendix:extra-example}

\begin{ex}\label{ex:k=2}
Consider $n=1$. If $k=2$ and $K =  \{ (x,y,0) \ / \ x,y \in  \{0,1 \}^2 \}  \approx \{0,1 \}^2$, %
 we have exactly two increasing functions of the kind %
 $s: \{ 0,1,2\} \to \{0,1\}^2 \approx K \subseteq Q$, i.e., $I=\{ s_1, s_2 \}$:
\begin{align*}
s_1: \{ 0,1,2\} &\to \{0,1\}^2              & s_2: \{ 0,1,2\} &\to \{0,1\}^2   \\
0 &\mapsto  (0,0),                                  & 0 &\mapsto (0,0)  \\
1 &\mapsto  (1,0),                                  & 1 &\mapsto (0,1) ,   \\
2 & \mapsto (1,1),                                  & 2 &\mapsto (1,1).
\end{align*}
This is not a surprise, since one needs two triangles to triangulate a square. From $s_1$ and $s_2$, we get two singular $2$-simplexes of the kind $\sigma : \Delta^2 \to 
 Q$:
\begin{align*}
\sigma_{ s_1 } : \quad \quad \quad \quad \ \Delta^2 &\to Q ,                       & \sigma_{ s_2 } : \quad \quad \quad \quad \ \Delta^2 &\to Q,\\
e_0=(1,0,0) &\mapsto (0,0,0) ,                                 & e_0=(1,0,0) &\mapsto  (0,0,0),           \\
e_1=(0,1,0) & \mapsto  (1,0,0) ,                               & e_1=(0,1,0) & \mapsto (0,1,0),              \\
e_2=(0,0,1)   & \mapsto (1,1,0) ,                              & e_2=(0,0,1)   & \mapsto (1,1,0).               
\end{align*}
The singular 2-chain of the triangulation is then
$$
\tau = (-1)^{s_1} \sigma_{{s_1}} + (-1)^{s_2} \sigma_{{s_2}} = \sigma_{{s_1}} - \sigma_{{s_2}}
$$
since
\begin{align*}
(-1)^{s_1} = sign
\begin{vmatrix}
s_1(1)-s_1(0)  \\
s_1(2)-s_1(0)
\end{vmatrix}
=sign
\begin{vmatrix}
1&0  \\
1&1
\end{vmatrix}
= +1
\end{align*}
and
\begin{align*}
(-1)^{s_2} = sign
\begin{vmatrix}
s_2(1)-s_2(0)  \\
s_2(2)-s_2(0)
\end{vmatrix}
=sign
\begin{vmatrix}
0&1  \\
1&1
\end{vmatrix}
= -1.
\end{align*}
So the orientation of the two linear simplexes is one the opposite of the other.  %
Finally the border of this triangulation is
\begin{align*}
\partial \tau = \partial \sigma_{s_1} - \partial \sigma_{s_2} &= 
 \sigma_{s_1} \circ F^0
-  \sigma_{s_1} \circ F^1
+  \sigma_{s_1} \circ F^2
-  \sigma_{s_2} \circ F^0
+  \sigma_{s_2} \circ F^1
-  \sigma_{s_2} \circ F^2\\
&=
 \sigma_{s_1} \circ F^0
+  \sigma_{s_1} \circ F^2
-  \sigma_{s_2} \circ F^0
-  \sigma_{s_2} \circ F^2
\end{align*}
because $  \sigma_{s_1} \circ F^1 =  \sigma_{s_2} \circ F^1 $. We see that $\partial \tau $ triangulates the border of the square with one simplex on each face.  %

Consider $n=1$ with $k=3$, we have a total of six maps in $I$ of the kind $s: \{ 0,1,2,3\} \to \{0,1\}^3 \subseteq Q$, i.e., $I=\{ s_1 ,\ s_2 ,\  s_3 ,\  s_4 ,\  s_5 ,\  s_6 \}$:
\begin{align*}
s_1: \{ 0,1,2,3\} &\to \{0,1\}^3              & s_2: \{ 0,1,2,3\} &\to \{0,1\}^3              & s_3: \{ 0,1,2,3\} &\to \{0,1\}^3   \\
 0 &\mapsto (0,0,0),               & 0 &\mapsto (0,0,0),               & 0 &\mapsto (0,0,0),  \\
 1 &\mapsto (1,0,0),               & 1 &\mapsto (1,0,0),               & 1 &\mapsto (0,1,0),    \\
 2 &\mapsto (1,1,0),               & 2 &\mapsto (1,0,1),               & 2 &\mapsto (0,1,1),    \\
 3 &\mapsto (1,1,1),               & 3 &\mapsto (1,1,1),               & 3 &\mapsto (1,1,1), 
\end{align*}
\begin{align*}
s_4: \{ 0,1,2,3\} &\to \{0,1\}^3              & s_5: \{ 0,1,2,3\} &\to \{0,1\}^3              & s_6: \{ 0,1,2,3\} &\to \{0,1\}^3   \\
 0 &\mapsto (0,0,0),               & 0 &\mapsto (0,0,0),               & 0 &\mapsto (0,0,0),  \\
 1 &\mapsto (0,1,0),               & 1 &\mapsto (0,0,1),               & 1 &\mapsto (0,0,1),    \\
 2 &\mapsto (1,1,0),               & 2 &\mapsto (1,0,1),               & 2 &\mapsto (0,1,1),    \\
 3 &\mapsto (1,1,1),               & 3 &\mapsto (1,1,1),               & 3 &\mapsto (1,1,1). 
\end{align*}
This matches our intuition that one needs six tetrahedra to triangulate the cube. From these functions, we get six singular $3$-simplexes of the kind $\sigma : \Delta^3 \to 
 Q$:
\begin{align*}
\sigma_{ s_1 } : \Delta^3 &\to Q ,      & \sigma_{ s_2 } : \Delta^3 &\to Q,       & \sigma_{ s_3 } : \Delta^3 &\to Q, \\
e_0   & \mapsto (0,0,0),               & e_0   & \mapsto (0,0,0),               & e_0   & \mapsto (0,0,0),  \\        
e_1   & \mapsto (1,0,0),               & e_1   & \mapsto (1,0,0),               & e_1   & \mapsto (0,1,0),  \\
e_2   & \mapsto (1,1,0),               & e_2   & \mapsto (1,0,1),               & e_2   & \mapsto (0,1,1),  \\
e_3   & \mapsto (1,1,1),               & e_3   & \mapsto (1,1,1),               & e_3   & \mapsto (1,1,1),
\end{align*}
\begin{align*}
\sigma_{ s_4 } : \Delta^3 &\to Q,      & \sigma_{ s_5 } : \Delta^3 &\to Q,       & \sigma_{ s_6 } : \Delta^3 &\to Q, \\
e_0   & \mapsto (0,0,0),               & e_0   & \mapsto (0,0,0),               & e_0   & \mapsto (0,0,0),  \\        
e_1   & \mapsto (0,1,0),               & e_1   & \mapsto (0,0,1),               & e_1   & \mapsto (0,0,1),  \\
e_2   & \mapsto (1,1,0),               & e_2   & \mapsto (1,0,1),               & e_2   & \mapsto (0,1,1),  \\
e_3   & \mapsto (1,1,1),               & e_3   & \mapsto (1,1,1),               & e_3   & \mapsto (1,1,1),
\end{align*}
The singular 3-chain of the triangulation is then
$$
\tau = \sum_{i=1}^6 (-1)^{s_i}  \sigma_{s_i} = \sigma_{s_1} -  \sigma_{s_2} +  \sigma_{s_3} -  \sigma_{s_4} +  \sigma_{s_5} -  \sigma_{s_6} 
$$
since
\begin{align*}
(-1)^{s_1} = sign
\begin{vmatrix}
s_1(1)- s_1(0)  \\
s_1(2)- s_1 (0)  \\
s_1(3)- s_1 (0)
\end{vmatrix}
=sign
\begin{vmatrix}
1&0&0  \\
1&1&0  \\
1&1&1
\end{vmatrix}
= +1,
\end{align*}
\begin{align*}
(-1)^{s_2} = sign
\begin{vmatrix}
s_2(1)- s_2(0)  \\
s_2(2)- s_2 (0)  \\
s_2(3)- s_2 (0)
\end{vmatrix}
=sign
\begin{vmatrix}
1&0&0  \\
1&0&1  \\
1&1&1
\end{vmatrix}
= -1
\end{align*}
and so on. %
The border of this triangulation is 
\begin{align*}
\partial \tau = \sum_{i=1}^6 (-1)^{s_i}  \sigma_{s_i} =& 
\partial \sigma_{s_1} - \partial \sigma_{s_2} + \partial \sigma_{s_3} - \partial \sigma_{s_4} + \partial \sigma_{s_5} - \partial \sigma_{s_6} \\
=& +
 \sigma_{s_1} \circ F^0
-  \sigma_{s_1} \circ F^1
+  \sigma_{s_1} \circ F^2
-  \sigma_{s_1} \circ F^3\\
&
-  \sigma_{s_2} \circ F^0
+  \sigma_{s_2} \circ F^1
-  \sigma_{s_2} \circ F^2
+  \sigma_{s_2} \circ F^3\\
&+ 
 \sigma_{s_3} \circ F^0
-  \sigma_{s_3} \circ F^1
+  \sigma_{s_3} \circ F^2
-  \sigma_{s_3} \circ F^3\\
&
-  \sigma_{s_4} \circ F^0
+  \sigma_{s_4} \circ F^1
-  \sigma_{s_4} \circ F^2
-  \sigma_{s_4} \circ F^3\\
&+ 
 \sigma_{s_5} \circ F^0
-  \sigma_{s_5} \circ F^1
+  \sigma_{s_5} \circ F^2
-  \sigma_{s_5} \circ F^3\\
&
-  \sigma_{s_6} \circ F^0
+  \sigma_{s_6} \circ F^1
-  \sigma_{s_6} \circ F^2
+  \sigma_{s_6} \circ F^3\\
=& +
 \sigma_{s_1} \circ F^0
-  \sigma_{s_1} \circ F^3
-  \sigma_{s_2} \circ F^0
+  \sigma_{s_2} \circ F^3\\
&+ 
 \sigma_{s_3} \circ F^0
-  \sigma_{s_3} \circ F^3
-  \sigma_{s_4} \circ F^0
+  \sigma_{s_4} \circ F^3\\
&+ 
 \sigma_{s_5} \circ F^0
-  \sigma_{s_5} \circ F^3
-  \sigma_{s_6} \circ F^0
+  \sigma_{s_6} \circ F^3
\end{align*}
because
$$
\begin{cases}
\sigma_{s_1} \circ F^2 = \sigma_{s_2} \circ F^2,\\
\sigma_{s_1} \circ F^1 = \sigma_{s_4} \circ F^1,\\
\sigma_{s_2} \circ F^1 = \sigma_{s_5} \circ F^1,\\
\sigma_{s_3} \circ F^2 = \sigma_{s_4} \circ F^2,\\
\sigma_{s_3} \circ F^1 = \sigma_{s_6} \circ F^1,\\
\sigma_{s_5} \circ F^2 = \sigma_{s_6} \circ F^2.
\end{cases}
$$
We see that $\partial \tau $ triangulates the border of the cube with two simplexes on each face.
\end{ex}


\section*{Declarations}
\noindent
This work was partially supported by a grant from the Vilho, Yrj\"{o} and Kalle V\"{a}is\"{a}l\"{a} Fund.


\bibliography{GC_manuscript}

\begin{thebibliography}{10}

\bibitem{ABB}
A.~Agrachev, D.~Barilari, and U.~Boscain.
\newblock {\em A Comprehensive Introduction to Sub-Riemannian Geometry}.
\newblock Cambridge University Press, 2019.

\bibitem{BKP}
Z.~M. Balogh, A.~Kozhevnikov, and P.~Pansu.
\newblock {H{\"o}lder Maps from Euclidean Spaces to Carnot Groups}.
\newblock preprint, 2017.

\bibitem{BB}
S.~Biagi and A.~Bonfiglioli.
\newblock {\em An Introduction to the Geometrical Analysis of Vector Fields:
  with Applications to Maximum Principles and Lie Groups}.
\newblock World Scientific Publishing Co. Pte. Ltd., 2019.

\bibitem{GCslicing}
G.~Canarecci.
\newblock {Sub-Riemannian Currents and Slicing of Currents in the Heisenberg
  group $\mathbb{H}^n$}.
\newblock {\em The Journal of Geometric Analysis}, 31:5166--5200, 2020.

\bibitem{CDPT}
L.~Capogna, D.~Danielli, S.~D. Pauls, and J.~T. Tyson.
\newblock {\em {An Introduction to the Heisenberg Group and the Sub-Riemannian
  Isoperimetric Problem}}.
\newblock Birkh{\"a}user Verlag AG, Basel - Boston - Berlin, 2007.

\bibitem{CG}
L.~Corwin and F.~P. Greenleaf.
\newblock {\em {Representations of nilpotent Lie groups and their applications.
  Part I: Basic theory and examples}}.
\newblock Cambridge University Press, Cambridge and New York, 1990.

\bibitem{ECHLPT}
D.~B.~A. Epstein, J.~W. Cannon, D.~F. Holt, S.~V.~F. Levy, M.~S. Paterson, and
  W.~P. Thurston.
\newblock {\em {Word Processing in Groups}}.
\newblock Taylor \& Francis Group, Boca Raton, London, New York, 1992.

\bibitem{FSSC}
B.~Franchi, R.~Serapioni, and F.~{Serra Cassano}.
\newblock {Regular Submanifolds, Graphs and Area Formula in Heisenberg Groups}.
\newblock {\em Advances in Mathematics}, 211(1):152--203, 2007.

\bibitem{GROMOVPDR}
M.~Gromov.
\newblock {\em {Partial Differential Relations}}.
\newblock Springer-Verlag, Berlin and Heidelberg, 1986.

\bibitem{GROMOV}
M.~Gromov.
\newblock {Carnot-Carath{\'e}odory Spaces Seen from Within}.
\newblock {\em Progress in Mathematics}, 144:79--323, 1996.

\bibitem{LEEtopological}
J.~M. Lee.
\newblock {\em {Introduction to Topological Manifolds}}.
\newblock Springer, New York, Dordrecht, Heidelberg, London, 2011.

\bibitem{MAU}
C.~R.~F. Maunder.
\newblock {\em {Algebraic Topology}}.
\newblock Cambridge University Press, Cambridge, 1980.

\bibitem{PANSU}
P.~Pansu.
\newblock {M{\'e}triques de Carnot--Carath{\'e}odory et Quasiisom{\'e}tries des
  Espaces Sym{\'e}triques de Rang Un}.
\newblock {\em Annals of Mathematics}, 129(1):1--60, 1989.

\bibitem{YOUNG2013}
R.~Young.
\newblock {Filling inequalities for nilpotent groups through approximation}.
\newblock {\em Groups, Geometry and Dynamics (EMS)}, (7):977--1011, 2013.

\bibitem{YOUNG}
R.~Young.
\newblock {High-Dimensional Fillings in Heisenberg Groups}.
\newblock {\em The Journal of Geometric Analysis}, 26(2):1596--1616, 2016.

\end{thebibliography}
\bibliographystyle{abbrv}

\end{document}